%% file: DTHMS-23.tex
\documentclass[11pt]{article}

\usepackage{pgf,acadpgf}
\usepackage{amssymb}
\usepackage{amsmath}
\usepackage{amsthm}
\usepackage[all]{xy}

\newcommand{\XYMATRIX}{\xymatrix@M=6pt}

\numberwithin{equation}{section}

\newtheoremstyle{slplain}% name
  {\topsep}%Space above
  {\topsep}%Space below
  {\slshape}%Body font
  {0pt}%Indent amount
  {\bfseries}% Theorem head font
  {.}%Punctuation after theorem head
  {0.5em}%Space after theorem head 2
  {}%Theorem head spec (can be left empty, meaning ‘normal’)

\theoremstyle{slplain}
  \newtheorem{THM}{Theorem}[section]
  \newtheorem{LEM}[THM]{Lemma}
  \newtheorem{PROP}[THM]{Proposition}
  \newtheorem{COR}[THM]{Corollary}

\theoremstyle{definition}
  \newtheorem{DEF}[THM]{Definition}
  \newtheorem{EX}{Example}[section]

\newcommand{\precsal}{\mathrel{\prec_{\mathit{sal}}}}
\newcommand{\precsaleq}{\mathrel{\preccurlyeq_{\mathit{sal}}}}

\newcommand{\alex}{\mathrel{<_{\mathit{alex}}}}

\newcommand{\sqalex}{\mathrel{\sqsubset_{\mathit{alex}}}}

\newcommand{\sqprimesal}{\mathrel{\sqsubset'_{\mathit{sal}}}}

\newcommand{\sqsal}{\mathrel{\sqsubset_{\mathit{sal}}}}
\newcommand{\sqsalidx}{\sqsubset_{\mathit{sal}}}
\newcommand{\lsal}{\mathrel{<_{\mathit{sal}}}}

\renewcommand{\preceq}{\preccurlyeq}

\renewcommand{\le}{\leqslant}
\renewcommand{\ge}{\geqslant}

\newcommand{\0}{\varnothing}
\newcommand{\quotient}[2]{\genfrac{[}{]}{0pt}{}{#1}{#2}}
\renewcommand{\sec}{\cap}
\renewcommand{\phi}{\varphi}
\renewcommand{\epsilon}{\varepsilon}

\newcommand{\CC}{\mathbf{C}}
\newcommand{\DD}{\mathbf{D}}

\newcommand{\KK}{\mathbf{K}}

\newcommand{\RR}{\mathbb{R}}
\newcommand{\TT}{\mathbf{T}}

\newcommand{\union}{\cup}
\newcommand{\restr}[2]{\hbox{$#1$}\hbox{$\upharpoonright$}_{#2}}
\newcommand{\reduct}[2]{\hbox{$#1$}\hbox{$|$}_{#2}}
\newcommand{\Boxed}[1]{\mbox{$#1$}}
\newcommand{\id}{\mathrm{id}}

\newcommand{\Ob}{\mathrm{Ob}}
\newcommand{\arity}{\mathrm{ar}}
\newcommand{\op}{\mathrm{op}}

\newcommand{\calA}{\mathcal{A}}
\newcommand{\calB}{\mathcal{B}}
\newcommand{\calC}{\mathcal{C}}
\newcommand{\calD}{\mathcal{D}}
\newcommand{\calE}{\mathcal{E}}
\newcommand{\calF}{\mathcal{F}}
\newcommand{\calG}{\mathcal{G}}

\newcommand{\calL}{\mathcal{L}}
\newcommand{\calM}{\mathcal{M}}

\newcommand{\calP}{\mathcal{P}}

\newcommand{\calS}{\mathcal{S}}
\newcommand{\calT}{\mathcal{T}}

\newcommand{\calX}{\mathcal{X}}
\newcommand{\calY}{\mathcal{Y}}

\newcommand{\REL}{\mathbf{Rel}}
\newcommand{\RELSrqm}{\mathbf{Rel}_{\mathit{srq}}}
\newcommand{\RslSrqm}{\mathbf{ERst}_{\mathit{srq}}}
\newcommand{\CH}{\mathbf{Ch}_{\mathit{emb}}}
\newcommand{\CHrs}{\mathbf{Ch}_{\mathit{rs}}}
\newcommand{\MetNers}{\mathbf{OMet}_{\mathit{ners}}}
\newcommand{\RBinRsh}{\mathbf{OBin}_{\mathit{rsh}}}
\newcommand{\GraSrqm}{\mathbf{OGra}_{\mathit{srq}}}
\newcommand{\RdlSrqm}{\mathbf{EDig}_{\mathit{srq}}}
\newcommand{\EPosSrqm}{\mathbf{EPos}_{\mathit{srq}}}
\newcommand{\ERhgSrqm}{\mathbf{OHgr}_{\mathit{srq}}}

\DeclareMathOperator{\tp}{tp}
\DeclareMathOperator{\mat}{mat}
\DeclareMathOperator{\tup}{tup}
\DeclareMathOperator{\Spec}{Spec}
\DeclareMathOperator{\Aut}{Aut}

\title{Dual Ramsey theorems for relational structures}
\author{%
  Dragan Ma\v sulovi\'c\\
  University of Novi Sad, Faculty of Sciences\\
  Department of Mathematics and Informatics\\
  Trg Dositeja Obradovi\'ca 3, 21000 Novi Sad, Serbia\\
  e-mail: dragan.masulovic@dmi.uns.ac.rs}

\begin{document}
\maketitle

\begin{abstract}
  In this paper we provide explicit dual Ramsey statements for several classes of finite relational structures
  (such as finite linearly ordered graphs, finite linearly ordered metric spaces and finite posets
  with a linear extension) and conclude the paper with an explicit dual of the Ne\v set\v ril-R\"odl Theorem
  for relational structures.
  Instead of embeddings which are crucial for ``direct'' Ramsey results, for each class of structures under
  consideration we propose a special class of surjective maps and
  prove a dual Ramsey theorem in such a setting. In contrast to on-going Ramsey classification projects
  where the research is focused on fine-tuning the objects, in this paper
  we advocate the idea that fine-tuning the morphisms is the key to proving dual Ramsey results.
  Since the setting we are interested in involves both structures and morphisms, all our results are spelled out using
  the reinterpretation of the (dual) Ramsey property in the language of category theory.

  \bigskip


  \noindent \textbf{AMS Subj.\ Classification (2010):} 05C55, 18A99
\end{abstract}

\section{Introduction}

Generalizing the classical results of F.~P.~Ramsey from the late 1920's, the structural Ramsey theory originated at
the beginning of 1970’s in a series of papers (see \cite{N1995} for references).
We say that a class $\KK$ of finite structures has the \emph{Ramsey property} if the following holds:
for any number $k \ge 2$ of colors and all $\calA, \calB \in \KK$ such that $\calA$ embeds into $\calB$
there is a $\calC \in \KK$
such that no matter how we color the copies of $\calA$ in $\calC$ with $k$ colors, there is a \emph{monochromatic} copy
$\calB'$ of $\calB$ in $\calC$ (that is, all the copies of $\calA$ that fall within $\calB'$ are colored by the same color).
In this parlance the Finite Ramsey Theorem takes the following form:

\begin{THM} (Finite Ramsey Theorem~\cite{Ramsey})\label{dthms.thm.frt}
  The class of all finite chains has the Ramsey property.
\end{THM}

Many natural classes of structures (such as finite graphs, metric spaces and posets, just to
name a few) do not have the Ramsey property. It is quite common,
though, that after expanding the structures under consideration
with appropriately chosen linear orders, the resulting class of expanded
structures has the Ramsey property. For example, the class of all finite
linearly ordered graphs $(V, E, \Boxed\sqsubset)$ where $(V, E)$ is a finite graph and $\sqsubset$ is
a linear order on the set $V$ of vertices of the graph has the Ramsey property~\cite{AH,Nesetril-Rodl}.
The same is true for metric spaces~\cite{Nesetril-metric}.
In case of finite posets we consider the class of all
finite linearly ordered posets $(P, \Boxed\preceq, \Boxed\sqsubset)$ where $(P, \Boxed\preceq)$ is
a finite poset and $\sqsubset$ is a linear order on $P$ which extends~$\preceq$~\cite{Nesetril-Rodl}.
One of the cornerstones of the structural Ramsey theory is the Ne\v set\v ril-R\"odl Theorem:

\begin{THM} (Ne\v set\v ril-R\"odl Theorem~\cite{AH}, \cite{Nesetril-Rodl,Nesetril-Rodl-1983})\label{dthms.thm.NRT}
  The class of all finite linearly ordered relational structures all having the same, fixed, relational type
  has the Ramsey property.\footnote{%
    Note that this is a restricted version of the
    Ne\v set\v ril-R\"odl Theorem which does not account for subclasses defined by forbidden substructures.%
  }
\end{THM}

The fact that this result has been proved independently by several research teams, and then reproved
in various ways and in various contexts \cite{AH,Nesetril-Rodl-1983,Nesetril-Rodl-1989,Promel-1985}
clearly demonstrates the importance and justifies the distinguished status
this result has in discrete mathematics.
The search for a dual version of the Ne\v set\v ril-R\"odl Theorem
was and still is an important research direction and several versions of the dual of the
Ne\v set\v ril-R\"odl Theorem have been published, most notably by
Spencer in~\cite{spencer},
Pr\"omel in~\cite{Promel-1985},
Pr\"omel and Voigt in~\cite{promel-voigt-sparse-GR},
Frankl, Graham, R\"odl in~\cite{frankl-graham-rodl} and recently by Solecki in~\cite{Solecki-2010}.

In cite~\cite{GR} Graham and Rothschild proved their famous Graham-Roth\-schild Theorem,
a powerful combinatorial statement about words intended for dealing with the Ramsey property of certain
geometric configurations. The fact that it also implies the following dual Ramsey statement was recognized
almost a decade later.

\begin{THM} (Finite Dual Ramsey Theorem~\cite{GR,Nesetril-Rodl-DRT})\label{dthms.thm.FDRT}
  For all positive integers $k$, $a$, $m$ there is a positive integer $n$ such that
  for every $n$-element set $C$ and every $k$-coloring of the set $\quotient Ca$ of all partitions of
  $C$ with exactly $a$ blocks there is a partition $\beta$ of $C$ with exactly $m$ blocks such that
  the set of all partitions from $\quotient Ca$ which are coarser than $\beta$ is monochromatic.
\end{THM}

It was observed in~\cite{promel-voigt-surj-sets} that each partition of a finite linearly ordered set can be
uniquely represented by the rigid surjection which takes each element of the underlying set to the minimum
of the block it belongs to (see Subsection~\ref{dthms.subset.lo} for the definition of a rigid surjection).
Hence, Finite Dual Ramsey Theorem is a structural Ramsey result about finite chains and special surjections between them.
This result was then generalized to trees in~\cite{solecki-dual-ramsey-trees}, and, using a different set of
techniques, to finite permutations in~\cite{masul-drp-perm}. In the setting of finite algebras,
Dual Ramsey theorems for finite boolean algebras and for finite distributive lattices endowed with a
particular linear order were proved in~\cite{masul-mudri}.

In this paper we provide explicit dual Ramsey statements for several classes of finite relational structures.
Instead of embeddings which are crucial for ``direct'' Ramsey results, for each class of structures under
consideration we propose a special class of surjective maps and
prove a dual Ramsey theorem for such a setting. In contrast to on-going Ramsey classification projects
(see for example~\cite{MANY-MANY}) where the research is focused on fine-tuning the objects, in this paper
we advocate the idea that fine-tuning the morphisms is the key to proving dual Ramsey results.
The basic setup of this paper relates strongly to~\cite{Promel-1985} where the
Ne\v set\v ril-R\"odl Theorem is interpreted in the language of category theory using the
concept of indexed categories. The main result of~\cite{Promel-1985}
is the partition theorem for combinatorial cubes.
In this sense it can be considered as a dual of the Ne\v set\v ril-R\"odl Theorem
(without forbidden substructures): objects are combinatorial cubes with selected
combinatorial subspaces and morphisms preserve the types of the selected subspaces. 
In this paper, however, we consider a dual of the Ne\v set\v ril-R\"odl Theorem spelled out
in the language of relational structures and base our approach on~\cite{masulovic-ramsey}
which can be thought of as a simplified version of the approach taken in~\cite{Promel-1985}.

In Section~\ref{dthms.sec.prelim} we give a brief overview of standard notions referring to
linear orders, total quasiorders and first order structures, and prove several technical results.

In Section~\ref{dthms.sec.rplct} we provide basics of category theory and give
a categorical reinterpretation of the Ramsey property
as proposed in~\cite{masulovic-ramsey}. We define the Ramsey property and the dual Ramsey property for
a category and illustrate these notions using some well-known examples. As our concluding example
we prove a dual Ramsey theorem for the category of finite linearly ordered metric spaces and non-expansive
rigid surjections.

Section~\ref{dthms.sec.homs-are-easy} is an \emph{en passant} display of several dual Ramsey theorems for
categories of structures and surjective rigid homomorphisms. All these results are immediate consequences of the
Finite Dual Ramsey Theorem because, as it turns out when dealing with such weak structure maps,
we can always take a sufficiently large empty structure to get the dual Ramsey property.

In Section~\ref{dthms.sec.qmaps} we turn to quotient maps in search of more challenging dual Ramsey results.
We prove dual Ramsey theorems for the following categories: 
the category $\RdlSrqm$ whose objects are finite reflexive digraphs with linear extensions and morphisms are special rigid quotient maps,
the category $\EPosSrqm$ whose objects are finite posets with linear extensions and morphisms are special rigid quotient maps,
the category $\ERhgSrqm(r)$, $r \ge 2$, whose objects are finite linearly ordered reflexive $r$-uniform hypergraphs and morphisms are special rigid quotient maps,
the category $\GraSrqm$ whose objects are finite linearly ordered reflexive graphs and morphisms are special rigid quotient maps,
and a few more subcategories of $\GraSrqm$.

Section~\ref{dthms.sec.DNRT} is devoted to proving yet another dual version
of the Ne\v set\v ril-R\"odl Theorem. We prove that
the class of all finite linearly ordered relational structures all having the same, fixed, relational type
has the dual Ramsey property with respect to a special class of rigid quotient maps.
Note again that this is a restricted formulation of the Ne\v set\v ril-R\"odl Theorem which does not account for
subclasses defined by forbidden ``quotients''.

The paper concludes with Section~\ref{dthms.sec.no-drp-tournaments} where
we prove that the category of finite linearly ordered reflexive tournaments
and rigid surjective homomorphisms does not have the dual Ramsey property.

\section{Preliminaries}
\label{dthms.sec.prelim}

In order to fix notation and terminology in this section we give a brief overview of standard notions referring to
linear orders, total quasiorders, first-order structures and category theory.

\subsection{Linear orders}
\label{dthms.subset.lo}

A \emph{chain} is a pair $\calA = (A, \Boxed\sqsubset)$ where $\sqsubset$ is a linear order on~$A$.
In case $A$ is finite we shall simply write
$\calA = \{a_1 \sqsubset a_2 \sqsubset \ldots \sqsubset a_n\}$.
Following \cite{promel-voigt-surj-sets} we say that a surjection
$
  f : \{a_1 \sqsubset a_2 \sqsubset \ldots \sqsubset a_n\} \to \{b_1 \sqsubset b_2 \sqsubset \ldots \sqsubset b_k\}
$
between two finite chains is \emph{rigid} if
$$
  \min f^{-1}(x) \sqsubset \min f^{-1}(y) \text{ whenever } x \sqsubset y.
$$
Equivalently, $f$ is rigid if for every $s \in \{1, \ldots, n\}$ there is a $t \in \{1, \ldots, k\}$ such that
$f(\{a_1, \ldots, a_s\}) = \{b_1, \ldots, b_t\}$. (In other words, a rigid surjection maps each initial segment
of $\{a_1 \sqsubset a_2 \sqsubset \ldots \sqsubset a_n\}$
onto an initial segment of $\{b_1 \sqsubset b_2 \sqsubset \ldots \sqsubset b_k\}$;
other than that, a rigid surjection is not required to respect the linear orders in question.)

Let $(A_i, \Boxed{\sqsubset_i})$ be finite chains, $1 \le i \le k$.
The linear orders $\sqsubset_i$, $1 \le i \le k$, induce the \emph{anti-lexicographic} order $\sqalex$
on $A_1 \times \ldots \times A_k$ by:
\begin{align*}
  (a_1, \ldots, a_k) \sqalex (b_1, \ldots, b_k) \text{ iff }
    & \text{there is an } s \in \{1, \ldots, k\} \text{ such that}\\
    & a_s \mathrel{\sqsubset_s} b_s, \text{ and } a_j = b_j \text{ for all } j > s.
\end{align*}
In particular, every finite chain $(A, \Boxed\sqsubset)$ induces the anti-lexicographic order $\sqalex$ on $A^n$, $n \ge 2$.

Let $\calA = (A, \Boxed\sqsubset)$ be a finite chain.
The linear order $\sqsubset$ induces the \emph{anti-lexicographic} order $\sqalex$
on $\calP(A)$ as follows. For $X \in \calP(A)$ let $\vec X \in \{0, 1\}^{|A|}$ denote the characteristic vector of $X$.
(As $A$ is linearly ordered, we can assign a string of 0's and 1's to each subset of $A$.)
Then for $X, Y \in \calP(A)$ we let:
$$
  X \sqalex Y \text{\ \ iff\ \ } \vec X \alex \vec Y,
$$
where $<$ is the usual ordering $0 < 1$. It is easy to see that for $X, Y \in \calP(A)$:
\begin{align*}
X \sqalex Y \text{ iff } X \subset Y, &\text{ or }\max\nolimits_\calA(X \setminus Y) \sqsubset \max\nolimits_\calA(Y \setminus X) \text{ in case}\\
                                           &\text{ $X$ and $Y$ are incomparable}.
\end{align*}

\subsection{Total quasiorders}

A \emph{total quasiorder} is a reflexive and transitive binary relation such that each pair of elements
of the underlying set is comparable. Each total quasiorder $\sigma$ on a set $I$ induces an equivalence relation $\equiv_\sigma$
on $I$ and a linear order $\prec_\sigma$ on $I / \Boxed{\equiv_\sigma}$ in a natural way: $i \mathrel{\equiv_\sigma} j$ if
$(i, j) \in \sigma$ and $(j, i) \in \sigma$, while $(i / \Boxed{\equiv_\sigma}) \mathrel{\prec_\sigma} (j / \Boxed{\equiv_\sigma})$
if $(i, j) \in \sigma$ and $(j, i) \notin \sigma$.
For the considerations that follow we need to linearly order all the total quasiorders on the same set, as follows.

\begin{DEF}
  Let $\sigma$ and $\tau$ be two distinct total quasiorders on $I = \{1, 2, \ldots, r\}$. Let
  \begin{align*}
    I / \Boxed{\equiv_\sigma} &= \{S_1 \alex S_2 \alex \ldots \alex S_k\}, \text{ and}\\
    I / \Boxed{\equiv_\tau} &= \{T_1 \alex T_2 \alex \ldots \alex T_\ell\}.
  \end{align*}
  (Here, $\alex$ stands for the anti-lexicographic ordering of $\calP(\{1, 2, \dots, r\})$ induced by the
  usual ordering of the integers.)
  
  We put $\sigma \mathrel{\triangleleft} \tau$ if $k < \ell$, or $k = \ell$ and
  $(S_1, S_2, \ldots, S_k) \mathrel{(\Boxed\alex)_{\mathit{alex}}} (T_1, T_2, \ldots, T_k)$.
  (Here, $(\Boxed\alex)_{\mathit{alex}}$ denotes the anti-lexicographic ordering on
  $\calP(\{1, 2, \dots, r\})^k$ induced by $\alex$ on $\calP(\{1, 2, \dots, r\})$.)
\end{DEF}

Let $(A, \Boxed\sqsubset)$ be a linearly ordered set, let $r$ be a positive integer, let $I = \{1, \ldots, r\}$ and
let $\overline a = (a_1, \ldots, a_r) \in A^r$. Then
$$
  \tp(\overline a) = \{(i, j) : a_i \sqsubseteq a_j \}
$$
is a total quasiorder on $I$ which we refer to as the \emph{type} of~$\overline a$.
Assume that $\sigma = \tp(\overline a)$. Let $s = |I / \Boxed{\equiv_\sigma}|$ and let $i_1$, \ldots, $i_s$
be the representatives of the classes of $\equiv_\sigma$ enumerated so that
$(i_1 / \Boxed{\equiv_\sigma}) \mathrel{\prec_\sigma} \ldots \mathrel{\prec_\sigma} (i_s / \Boxed{\equiv_\sigma})$.
Then
$$
  \mat(\overline a) = (a_{i_1}, \ldots, a_{i_s})
$$
is the \emph{matrix of $\overline a$}. Note that $a_{i_1} \sqsubset \ldots \sqsubset a_{i_s}$.

For a total quasiorder $\sigma$ on $I$ such that
$|I / \Boxed{\equiv_\sigma}| = s$ and an arbitrary
$s$-tuple $\overline b = (b_1, \ldots, b_s) \in A^s$
define an $r$-tuple
$$
  \tup(\sigma, \overline b) = (a_1, \ldots, a_r) \in A^r
$$
as follows. Let $i_1$, \ldots, $i_s$ be the representatives of the classes of $\equiv_\sigma$ enumerated so that
$(i_1 / \Boxed{\equiv_\sigma}) \mathrel{\prec_\sigma} \ldots \mathrel{\prec_\sigma} (i_s / \Boxed{\equiv_\sigma})$.
Then put
$$
  a_\eta = b_\xi \text{ if and only if } \eta \mathrel{\equiv_\sigma} i_\xi.
$$
(In other words, we put $b_1$ on all the entries in $i_1 / \Boxed{\equiv_\sigma}$, we put
$b_2$ on all the entries in $i_2 / \Boxed{\equiv_\sigma}$, and so on.)

It is a matter of routine to check that
for every tuple $\overline a$ and every tuple $\overline b = (b_1, b_2, \ldots, b_s)$
such that $b_1 \sqsubset b_2 \sqsubset \ldots \sqsubset b_s$:
\begin{equation}\label{dthms.eq.tup-mat}
  \begin{aligned}
    \tp(\tup(\sigma, \overline b)) = \sigma, \text{ } \mat(\tup(\sigma, \overline b))    &= \overline b, \text{ and}\\
    \tup(\tp(\overline a), \mat(\overline a)) &= \overline a,\\
  \end{aligned}
\end{equation}

\begin{DEF}
  Let $(A, \Boxed\sqsubset)$ be a finite chain and let $n \ge 2$.
  Define the linear order $\sqsal$ on $A^n$ as follows (``sal'' in the subscript stands for
  ``special anti-lexicographic''). Take any $\overline a, \overline b \in A^n$ where
  $\overline a = (a_1, a_2, \ldots, a_n)$ and $\overline b = (b_1, b_2, \ldots, b_n)$.
  \begin{itemize}
  \item
    If $\tp(\overline a) \triangleleft \tp(\overline b)$ then $\overline a \sqsal \overline b$;
  \item
    if $\tp(\overline a) = \tp(\overline b)$ and $\{a_1, a_2, \ldots, a_n\} \ne \{b_1, b_2, \ldots, b_n\}$ then
    $\overline a \sqsal \overline b \text{ iff } \{a_1, a_2, \ldots, a_n\} \sqalex \{b_1, b_2, \ldots, b_n\}$.
  \end{itemize}
  (Note that $\tp(\overline a) = \tp(\overline b)$ and $\{a_1, a_2, \ldots, a_n\} = \{b_1, b_2, \ldots, b_n\}$
  imply~$\overline a = \overline b$.)
\end{DEF}

\begin{LEM}\label{dthm.lem.tp-mat-facts}
  Let $(A, \Boxed\sqsubset)$ be a finite chain and let $n \ge 2$ be an integer.
  
  $(a)$ For all $\overline a, \overline b \in A^n$ we have that
  $\overline a = \overline b$ iff $\mat(\overline a) = \mat(\overline b)$ and $\tp(\overline a) = \tp(\overline b)$.
  
  $(b)$ Assume that $\tp(\overline a) = \tp(\overline b)$ for some $\overline a, \overline b \in A^n$. Then
  $\overline a \sqsal \overline b$ iff $\mat(\overline a) \sqsal \mat(\overline b)$.
  
  $(c)$ Assume that $\tp(\overline a^1) = \tp(\overline a^2) = \ldots = \tp(\overline a^k)$ for some
  $\overline a^1, \overline a^2, \ldots, \overline a^k \in A^n$. Then
  $\min\{\mat(\overline a^1), \mat(\overline a^2), \ldots, \mat(\overline a^k)\} =
  \mat(\min\{\overline a^1, \overline a^2, \ldots, \overline a^k\})$,
  where both minima are taken with respect to $\sqsal$.
\end{LEM}
\begin{proof}
  $(a)$ is obvious, and $(c)$ follows directly from $(b)$. So, let us show $(b)$.
  Let $\overline a = (a_1, a_2, \ldots, a_n)$ and $\overline b = (b_1, b_2, \ldots, b_n)$.
  If $\{a_1, a_2, \ldots, a_n\} = \{b_1, b_2, \ldots, b_n\}$ then
  $\tp(\overline a) = \tp(\overline b)$ implies $\overline a = \overline b$.
  Assume, therefore, that $\{a_1, a_2, \ldots, a_n\} \ne \{b_1, b_2, \ldots, b_n\}$. Let
  $\mat(\overline a) = (a_{i_1}, a_{i_2}, \ldots, a_{i_r})$ for some indices
  $i_1, i_2, \ldots, i_r$. Then $\tp(\overline a) = \tp(\overline b)$ implies that
  $\mat(\overline b) = (b_{i_1}, b_{i_2}, \ldots, b_{i_r})$. Note, also, that
  $\{a_1, a_2, \ldots, a_n\} = \{a_{i_1}, a_{i_2}, \ldots, a_{i_r}\}$ and that
  $\{b_1, b_2, \ldots, b_n\} = \{b_{i_1}, b_{i_2}, \ldots, b_{i_r}\}$. Therefore,
  $\overline a \sqsal \overline b$ iff $\{a_1, a_2, \ldots, a_n\} \sqalex \{b_1, b_2, \ldots, b_n\}$
  iff $\{a_{i_1}, a_{i_2}, \ldots, a_{i_r}\} \sqalex \{b_{i_1}, b_{i_2}, \ldots, b_{i_r}\}$ iff
  $\mat(\overline a) \sqsal \mat(\overline b)$.
\end{proof}

\begin{LEM}\label{dthm.lem.tp-mat-facts-2}
  Let $(A, \Boxed\sqsubset)$ and $(B, \Boxed{\sqsubset'})$ be finite chains. For every $n \ge 2$ and every mapping
  $f : A \to B$ define $\hat f : A^n \to B^n$ by
  $$
    \hat f(a_1, a_2, \ldots, a_n) = (f(a_1), f(a_2), \ldots, f(a_n)).
  $$

  $(a)$ For every total quasiorder $\sigma$ such that $|A / \Boxed{\equiv_\sigma}| = n$ and every
  $\overline a \in A^n$ we have that $\tup(\sigma, \hat f(\overline a)) = \hat f(\tup(\sigma, \overline a))$.
  
  $(b)$ Take any $\overline a = (a_1, a_2, \ldots, a_n) \in A^n$ and assume that
  $a_i \mathrel{\sqsubset} a_j \Rightarrow f(a_i) \mathrel{\sqsubset'} f(a_j)$ for all $i$ and $j$.
  Then $\tp(\hat f(\overline a)) = \tp(\overline a)$ and $\mat(\hat f(\overline a)) = \hat f(\mat(\overline a))$.
\end{LEM}
\begin{proof}
  Both $(a)$ and $(b)$ are straightforward.
\end{proof}

\begin{LEM}\label{dthms.lem.AUX}
  Let $(A, \Boxed\sqsubset)$ and $(B, \Boxed{\sqsubset'})$ be finite chains. Let $n \ge 2$ be an integer and let
  $\theta \subseteq A^n$ and $\theta' \subseteq B^n$ be relations. Let $\sigma$ be a total quasiorder on
  $\{1, 2, \ldots, n\}$, let $r = |\{1, 2, \ldots, n\} / \Boxed{\equiv_\sigma}| \ge 2$ and let
  \begin{align*}
    \rho  &= \{\mat(\overline x) : \overline x \in \theta \text{ and } \tp(\overline x) = \sigma\} \subseteq A^r, \text{ and}\\
    \rho' &= \{\mat(\overline x) : \overline x \in \theta' \text{ and } \tp(\overline x) = \sigma\} \subseteq B^r.
  \end{align*}
  Furthermore, let $f : A \to B$ be a mapping such that:
  \begin{itemize}
  \item
    for every $(x_1, x_2, \ldots, x_n) \in \theta$,
    if $\restr{f}{\{x_1, x_2, \ldots, x_n\}}$ is not a constant map
    then $x_i \mathrel{\sqsubset} x_j \Rightarrow
    f(x_i) \mathrel{\sqsubset'} f(x_j)$ for all $i$ and $j$, and
  \item
    $\restr{\hat f}{\theta} : \theta \to \theta'$ is well defined (that is, $\hat f(\overline x) \in \theta'$ for all $\overline x \in \theta$)
    and surjective.
  \end{itemize}
  Then $\restr{\hat f}{\rho} : \rho \to \rho'$ is well defined, surjective, and for all $\overline p \in \theta'$
  such that $\tp(\overline p) = \sigma$ we have that
  $\min(\restr{\hat f}{\rho}^{-1}(\mat(\overline p))) = \mat(\min \restr{\hat f}{\theta}^{-1}(\overline p))$,
  where both minima are taken with respect to~$\sqsal$.
\end{LEM}
\begin{proof}
  For notational convenience let $\hat f_\theta = \restr{\hat f}{\theta}$ and $\hat f_\rho = \restr{\hat f}{\rho}$.
  Lemmas~\ref{dthm.lem.tp-mat-facts} and~\ref{dthm.lem.tp-mat-facts-2} ensure that
  $\hat f_\rho : \rho \to \rho'$ is well defined and surjective. Take any $\overline p \in \theta'$
  such that $\tp(\overline p) = \sigma$ and let
  $\hat f_\theta^{-1}(\overline p) = \{\overline a^1, \overline a^2, \ldots, \overline a^k\}$.
  Let us first show that $\hat f_\rho^{-1}(\mat(\overline p)) = \{\mat(\overline a^1), \mat(\overline a^2), \ldots,
  \mat(\overline a^k)\}$.
  
  $(\supseteq)$ Take any~$i \in \{1, 2, \ldots, k\}$. Let us first show that $\tp(\overline a^i) = \sigma$
  for all~$i$. Since $\hat f(\overline a^i) = \overline p$ and $r \ge 2$ it follows that
  $\restr{f}{\{a^i_1, a^i_2, \ldots, a^i_n\}}$ is not a constant map. Then by the assumption
  (the first bullet above) we have that $a^i_s \mathrel{\sqsubset} a^i_t \Rightarrow
  f(a^i_s) \mathrel{\sqsubset'} f(a^i_t)$ for all $s$ and $t$.
   now yields that
  $\tp(\overline a^i) = \tp(\hat f(\overline a^i))$. Since $\hat f(\overline a^i) =
  \hat f_\theta(\overline a^i) = \overline p$, it follows that
  $\tp(\overline a^i) = \tp(\overline p) = \sigma$.
  Applying Lemma~\ref{dthm.lem.tp-mat-facts-2}~$(b)$ once again gives
  $\hat f_\rho(\mat(\overline a^i)) = \mat(\hat f_\theta(\overline a^i)) = \mat(\overline p)$.
  
  $(\subseteq)$ Take any $\overline u \in \hat f_\rho^{-1}(\mat(\overline p))$.
  Then $\overline u = \mat(\overline x)$ for some
  $\overline x \in \theta$ such that $\tp(\overline x) = \sigma$, so
  $\mat(\overline p) = \hat f_\rho(\overline u) = \hat f_\rho(\mat(\overline x)) = \mat(\hat f_\theta(\overline x))$
  (Lemma~\ref{dthm.lem.tp-mat-facts-2}).
  On the other hand, $\tp(\overline p) = \sigma = \tp(\overline x) = \tp(\hat f_\theta(\overline x))$
  using Lemma~\ref{dthm.lem.tp-mat-facts-2} once more.
  So, $\mat(\overline p) = \mat(\hat f_\theta(\overline x))$ and $\tp(\overline p) = \tp(\hat f_\theta(\overline x))$
  whence, by Lemma~\ref{dthm.lem.tp-mat-facts}, $\overline p = \hat f_\theta(\overline x)$. In other words,
  $\overline x \in \hat f_\theta^{-1}(\overline p)$ whence $\overline x = \overline a^i$ for some~$i$.
  Then $\overline u = \mat(\overline x) = \mat(\overline a^i)$ for some~$i$.
  
  \medskip
  
  Let us now show that $\min(\hat f_\rho^{-1}(\mat(\overline p))) = \mat(\min \hat f_\theta^{-1}(\overline p))$.
  Lemma~\ref{dthm.lem.tp-mat-facts-2} yields that
  $\tp(\overline a^1) = \tp(\overline a^2) = \ldots = \tp(\overline a^k)$ so by
  Lemma~\ref{dthm.lem.tp-mat-facts} we have that
  \begin{align*}
    \min(\hat f_\rho^{-1}(\mat(\overline p)))
      &= \min\{\mat(\overline a^1), \mat(\overline a^2), \ldots, \mat(\overline a^k)\}\\
      &= \mat(\min\{\overline a^1, \overline a^2, \ldots, \overline a^k\})\\
      &= \mat(\min \hat f_\theta^{-1}(\overline p)).
  \end{align*}
  This concludes the proof.
\end{proof}

\subsection{Structures}

Let $\Theta$ be a set of relational symbols.
A \emph{$\Theta$-structure} $\calA = (A, \Theta^\calA)$ is a set $A$ together with a set $\Theta^\calA$ of
relations on $A$ which are interpretations of the corresponding symbols in $\Theta$.
The interpretation of a symbol $\theta \in \Theta$ in the $\Theta$-structure $\calA$ will be denoted by $\theta^\calA$.
The underlying set of a structure $\calA$, $\calB$, $\calC$, \ldots\ will always be denoted by its roman letter $A$, $B$, $C$, \ldots\ respectively.
A structure $\calA = (A, \Theta^\calA)$ is \emph{finite} if $A$ is a finite set.
A structure $\calA = (A, \Theta^\calA)$ is \emph{reflexive} if the following holds for every $\theta \in \Theta$, where
$r = \arity(\theta)$:
$$
  \Delta_{A, r} = \{(\,\underbrace{a, a, \ldots, a}_{r}\,) : a \in A\} \subseteq \theta^\calA.
$$
In case of reflexive structures unary relations play no role because every reflexive unary relation is trivial.

Let $\calA$ and $\calB$ be $\Theta$-structures. A mapping $f : A \to B$
is a \emph{homomorphism} from $\calA$ to $\calB$, and we write $f : \calA \to \calB$, if
$$
  (a_1, \ldots, a_r) \in \theta^\calA \Rightarrow (f(a_1), \ldots f(a_r)) \in \theta^\calB,
$$
for each $\theta \in \Theta$ and $a_1, \ldots, a_r \in A$.
A homomorphism $f : \calA \to \calB$ is an \emph{embedding} if $f$ is injective and
$$
  (f(a_1), \ldots f(a_r)) \in \theta^\calB \Leftrightarrow (a_1, \ldots, a_r) \in \theta^\calA,
$$
for each $\theta \in \Theta$ and $a_1, \ldots, a_r \in A$. A homomorphism $f : \calA \to \calB$ is a
\emph{quotient map} if $f$ is surjective and for every $\theta \in \Theta$ and $(b_1, \ldots, b_r) \in \theta^\calB$
there exists an $(a_1, \ldots, a_r) \in \theta^\calA$ such that $f(a_i) = b_i$, $1 \le i \le r$.

Let $\Theta$ be a relational language and let $\Boxed\sqsubset \notin \Theta$ be a binary relational symbol.
A \emph{linearly ordered $\Theta$-structure} is a
$(\Theta \union \{\Boxed\sqsubset\})$-structure $\calA = (A, \Theta^\calA, \Boxed{\sqsubset^\calA})$
where $(A, \Theta^\calA)$ is a $\Theta$-structure and $\sqsubset^\calA$ is a linear order on~$A$.
A linearly ordered $\Theta$-structure $\calA = (A, \Theta^\calA, \Boxed{\sqsubset^\calA})$ is \emph{reflexive}
if $(A, \Theta^\calA)$ is a reflexive $\Theta$-structure.

\section{Category theory and the Ramsey property}
\label{dthms.sec.rplct}

In order to specify a \emph{category} $\CC$ one has to specify
a class of objects $\Ob(\CC)$, a set of morphisms $\hom_\CC(A, B)$ for all $A, B \in \Ob(\CC)$,
the identity morphism $\id_A$ for all $A \in \Ob(\CC)$, and
the composition of mor\-phi\-sms~$\cdot$~so that
$\id_B \cdot f = f \cdot \id_A = f$ for all $f \in \hom_\CC(A, B)$, and
$(f \cdot g) \cdot h = f \cdot (g \cdot h)$ whenever the compositions are defined.
A morphism $f \in \hom_\CC(B, C)$ is \emph{monic} or \emph{left cancellable} if
$f \cdot g = f \cdot h$ implies $g = h$ for all $g, h \in \hom_\CC(A, B)$ where $A \in \Ob(\CC)$ is arbitrary.
A morphism $f \in \hom_\CC(B, C)$ is \emph{epi} or \emph{right cancellable} if
$g \cdot f = h \cdot f$ implies $g = h$ for all $g, h \in \hom_\CC(C, D)$ where $D \in \Ob(\CC)$ is arbitrary.
A morphism $f \in \hom_\CC(A, B)$ is \emph{invertible} if there exists a morphism $g \in \hom_\CC(B, A)$
such that $g \cdot f = \id_A$ and $f \cdot g = \id_B$.
Let $\Aut_\CC(A)$ denote the set of all the invertible morphisms in $\hom_\CC(A, A)$.
An object $A \in \Ob(\CC)$ is \emph{rigid} if $\Aut_\CC(A) = \{\id_A\}$.

\begin{EX}
  Let $\CH$ denote the the category whose objects are finite chains
  and whose morphisms are embeddings.
\end{EX}

\begin{EX}
  By $\REL(\Theta, \Boxed\sqsubset)$ we denote the category whose objects are finite linearly ordered
  $\Theta$-structures and whose morphisms are embeddings.
\end{EX}

\begin{EX}\label{drpperm.ex.CHrs-def}
  The composition of two rigid surjections is again a rigid surjection, so finite chains and rigid surjections constitute a category
  which we denote by~$\CHrs$.
\end{EX}

For a category $\CC$, the \emph{opposite category}, denoted by $\CC^\op$, is the category whose objects
are the objects of $\CC$, morphisms are formally reversed so that
$
  \hom_{\CC^\op}(A, B) = \hom_\CC(B, A)
$,
and so is the composition:
$
  f \cdot_{\CC^\op} g = g \cdot_\CC f
$.

A category $\DD$ is a \emph{subcategory} of a category $\CC$ if $\Ob(\DD) \subseteq \Ob(\CC)$ and
$\hom_\DD(A, B) \subseteq \hom_\CC(A, B)$ for all $A, B \in \Ob(\DD)$.
A category $\DD$ is a \emph{full subcategory} of a category $\CC$ if $\Ob(\DD) \subseteq \Ob(\CC)$ and
$\hom_\DD(A, B) = \hom_\CC(A, B)$ for all $A, B \in \Ob(\DD)$.

A \emph{functor} $F : \CC \to \DD$ from a category $\CC$ to a category $\DD$ maps $\Ob(\CC)$ to
$\Ob(\DD)$ and maps morphisms of $\CC$ to morphisms of $\DD$ so that
$F(f) \in \hom_\DD(F(A), F(B))$ whenever $f \in \hom_\CC(A, B)$, $F(f \cdot g) = F(f) \cdot F(g)$ whenever
$f \cdot g$ is defined, and $F(\id_A) = \id_{F(A)}$.

Categories $\CC$ and $\DD$ are \emph{isomorphic} if there exist functors $F : \CC \to \DD$ and $G : \DD \to \CC$ which are
inverses of one another both on objects and on morphisms.

The \emph{product} of categories $\CC_1$ and $\CC_2$ is the category $\CC_1 \times \CC_2$ whose objects are pairs $(A_1, A_2)$
where $A_1 \in \Ob(\CC_1)$ and $A_2 \in \Ob(\CC_2)$, morphisms are pairs $(f_1, f_2) : (A_1, A_2) \to (B_1, B_2)$ where
$f_1 : A_1 \to B_1$ is a morphism in $\CC_1$ and $f_2 : A_2 \to B_2$ is a morphism in $\CC_2$.
The composition of morphisms is carried out componentwise: $(f_1, f_2) \cdot (g_1, g_2) = (f_1 \cdot g_1, f_2 \cdot g_2)$.

Let $\CC$ be a category and $\calS$ a set. We say that
$
  \calS = \calX_1 \union \ldots \union \calX_k
$
is a \emph{$k$-coloring} of $\calS$ if $\calX_i \sec \calX_j = \0$ whenever $i \ne j$.
Equivalently, a $k$-coloring of $\calS$ is any map $\chi : \calS \to \{1, 2, \ldots, k\}$.
For an integer $k \ge 2$ and $A, B, C \in \Ob(\CC)$ we write
$
  C \longrightarrow (B)^{A}_k
$
to denote that for every $k$-coloring
$
  \hom_\CC(A, C) = \calX_1 \union \ldots \union \calX_k
$
there is an $i \in \{1, \ldots, k\}$ and a morphism $w \in \hom_\CC(B, C)$ such that
$w \cdot \hom_\CC(A, B) \subseteq \calX_i$.

\begin{DEF}
  A category $\CC$ has the \emph{Ramsey property} if
  for every integer $k \ge 2$ and all $A, B \in \Ob(\CC)$
  such that $\hom_\CC(A, B) \ne \0$ there is a
  $C \in \Ob(\CC)$ such that $C \longrightarrow (B)^{A}_k$.
  
  A category $\CC$ has the \emph{dual Ramsey property} if $\CC^\op$ has the Ramsey property.
\end{DEF}

Clearly, if $\CC$ and $\DD$ are isomorphic categories and one of them has the (dual) Ramsey property,
then so does the other. Actually, even more is true: if $\CC$ and $\DD$ are equivalent categories
and one of them has the (dual) Ramsey property, then so does the other. We refrain from providing the
definition of (the fairly standard notion of) categorical equivalence as we shall have no use for it
in this paper, and for the proof we refer the reader to~\cite{masulovic-ramsey}.

\begin{EX}\label{cerp.ex.FRT-ch}
  The category $\CH$ of finite chains and embeddings
  has the Ramsey property. This is just a reformulation of the Finite Ramsey Theorem (Theorem~\ref{dthms.thm.frt}).
\end{EX}

\begin{EX}
  The category $\REL(\Theta, \Boxed\sqsubset)$ has the Ramsey property. This is the famous
  Ne\v set\v ril-R\"odl Theorem (Theorem~\ref{dthms.thm.NRT}).
\end{EX}

\begin{EX}\label{cerp.ex.FDRT-ch}
  The category $\CHrs$ of finite chains and rigid surjections
  has the dual Ramsey property. This is just a reformulation of the Finite Dual Ramsey Theorem (Theorem~\ref{dthms.thm.FDRT};
  see also the discussion in the Introduction.)
\end{EX}

One of the main benefits of considering the Ramsey property in the setting of category theory is
the Duality Principle which is a metatheorem of category theory:

\begin{quote}
  \textbf{The Duality Principle.}
  \textsl{If a statement $\phi$ is true in a category $\CC$, then the opposite statement $\phi^\op$ is true in~$\CC^\op$.}
\end{quote}

For a detailed technical discussion and the precise definition of $\phi^\op$ we refer the reader to~\cite{AHS}.
Here, however, we would like to stress that the Duality Principle not only allows us to present
an elegant definition of the dual Ramsey property
but actually saves quite a lot of work, in particular in situations where we want to reuse the
existing Ramsey-type results to infer dual Ramsey-type results. For example,
in \cite[Proposition~2.3]{masulovic-ramsey} we proved that if
$\CC$ is a category where morphisms are monic and $\CC$
has the Ramsey property then all the objects in $\CC$ are rigid.
As an immediate consequence of the Duality Principle we have the following
(bearing in mind that rigidity is a self-dual notion):

\begin{COR}\label{dthms.cor.rigid}
  Let $\CC$ be a category where morphisms are epi. If $\CC$
  has the dual Ramsey property then all the objects in $\CC$ are rigid.
\end{COR}

The following simple technical result will be useful later.

\begin{LEM}\label{dthms.lem.1}
  Let $\CC$ be a category,
  let $A, B, C, D \in \Ob(\CC)$ be arbitrary and let $k \ge 2$ be an integer.
  If $C \longrightarrow (B)^A_k$ and $\hom(C, D) \ne \0$ then $D \longrightarrow (B)^A_k$.
\end{LEM}
\begin{proof}
  Fix an arbitrary $s \in \hom(C, D)$. Let 
  $
    \hom(A, D) = \calX_1 \union \ldots \union \calX_k
  $
  be a $k$-coloring and let 
  $
    \calY_i = \{f \in \hom(A, C) : s \cdot f \in \calX_i\}
  $.
  It is easy to see that
  $
    \hom(A, C) = \calY_1 \union \ldots \union \calY_k
  $
  is a $k$-coloring.
  Since $C \longrightarrow (B)^A_k$, there is an $i \in \{1, \ldots, k\}$ and a morphism $w \in \hom(B, C)$ such that
  $w \cdot \hom(A, B) \subseteq \calY_i$. But then
  $s \cdot w \cdot \hom(A, B) \subseteq s \cdot \calY_i \subseteq \calX_i$.
\end{proof}

As our concluding example we shall prove the dual Ramsey theorem for
linearly ordered metric spaces. A \emph{linearly ordered metric space} is a triple $\calM = (M, d, \Boxed\sqsubset)$
where $d : M^2 \to \RR$ is a metric on $M$ and $\sqsubset$ is a linear order on~$M$.
For a positive integer $n$ and a positive real number $\delta$ let
$$
  \calM_{n,\delta} = (\{1, 2, \ldots, n\}, d_n^\delta, \Boxed<)
$$
be the linearly ordered metric space where $<$ is the usual ordering of the integers and $d_n^\delta(x, y) = \delta$ whenever $x \ne y$.
A mapping $f : M \to M'$ is a
\emph{non-expansive rigid surjection} from $(M, d, \Boxed\sqsubset)$ to $(M', d', \Boxed{\sqsubset'})$ if
\begin{itemize}
\item
  $f : (M, d) \to (M', d')$ is non-expansive, that is, $d'(f(x), f(y)) \le d(x, y)$ for all $x, y \in M$, and
\item
  $f : (M, \Boxed\sqsubset) \to (M', \Boxed{\sqsubset'})$ is a rigid surjection.
\end{itemize}
Let $\MetNers$ be the category whose objects are finite
linearly ordered metric spaces and morphisms are non-expansive rigid surjections.

For a linearly ordered metric space $\calM = (M, d, \Boxed\sqsubset)$ let
$\Spec(\calM) = \{d(x, y) : x, y \in M \text{ and } x \ne y\}$. For a subcategory $\CC$ of $\MetNers$ let
$$
  \Spec(\CC) = \bigcup \{\Spec(\calM) : \calM \in \Ob(\CC)\}.
$$

\begin{THM}\label{dthms.thm.met}
  Let $\CC$ be a full subcategory of $\MetNers$ such that for every positive integer $m$ and every $\delta \in \Spec(\CC)$
  there is an integer $n \ge m$ such that $\calM_{n,\delta} \in \Ob(\CC)$. Then
  $\CC$ has the dual Ramsey property. In particular, the following categories have the dual Ramsey property:
  \begin{itemize}
  \item
    the category $\MetNers$;

  \item
    the category $\MetNers(S)$, $S \subseteq \RR$, which stands for the full subcategory of $\MetNers$ spanned
    by all those finite linearly ordered metric spaces $\calM$ such that $\Spec(\calM) \subseteq S$.
  \end{itemize}
\end{THM}
\begin{proof}
  Take any $k \ge 2$, any $\calA = (A, d^\calA, \Boxed{\sqsubset^\calA})$ and $\calB  = (B, d^\calB, \Boxed{\sqsubset^\calB})$ in
  $\CC$ such that there is a non-expansive rigid surjection $\calB \to \calA$,
  and let $(C, \Boxed{\sqsubset^\calC})$ be a finite chain such that
  $(C, \Boxed{\sqsubset^\calC}) \longrightarrow (B, \Boxed{\sqsubset^\calB})^{(A, \Boxed{\sqsubset^\calA})}_k$ in $\CHrs^\op$.
  Such a chain exists because $\CHrs^\op$ has the Ramsey property (Example~\ref{cerp.ex.FDRT-ch}).
  Let
  $
    \delta = \max\{d^\calB(x, y) : x, y \in B\}
  $.
  By the assumption, there is an integer $n \ge |C|$ such that $\calM_{n,\delta} \in \Ob(\CC)$.

  Since $n \ge |C|$, there is a rigid surjection $(\{1, 2, \ldots, n\}, \Boxed<) \to (C, \Boxed{\sqsubset^\calC})$,
  so Lemma~\ref{dthms.lem.1} yields that
  $
    (\{1, 2, \ldots, n\}, \Boxed<) \longrightarrow (B, \Boxed{\sqsubset^\calB})^{(A, \Boxed{\sqsubset^\calA})}_k
  $
  in $\CHrs^\op$. Then it is easy to show that
  $
    \calM_{n,\delta} \longrightarrow (\calB)^\calA_k
  $
  in $\CC^\op$ because $f : \calM_{n,\delta} \to \calA$ is a non-expansive rigid surjection
  if and only if $f : (\{1, 2, \ldots, n\}, \Boxed{<}) \to (A, \Boxed{\sqsubset^\calA})$ is a rigid surjection.
\end{proof}

\paragraph{What is an acceptable kind of objects?}
Having in mind Corollary~\ref{dthms.cor.rigid},
a necessary requirement for a category to have the dual Ramsey property is that all of its objects be rigid.
In this paper we consider categories of finite linearly ordered relational structures, as adding linear orders
to finite structures turns out to be technically the easiest way of achieving rigidity. As we shall see
in an instant, the morphisms we will be working with will be surjective, so
all the structures in the paper will necessarily be reflexive.

\paragraph{What is an acceptable kind of morphisms?}
Embeddings have established themselves as the only kind of morphisms of interest
when considering ``direct'' Ramsey results in structural Ramsey theory.
For dual Ramsey results, though, there is no obvious notion that parallels in full the notion of embedding.
For example, fix a relational language $\Theta$ and consider a category $\CC$ whose objects
are some finite linearly ordered $\Theta$-structures $\calA = (A, \Theta^\calA, \Boxed{<^\calA})$
and morphisms are just rigid surjections $f : (A, \Boxed{<^\calA}) \to (B, \Boxed{<^\calB})$.
Then $\CC$ has the dual Ramsey property provided it contains arbitrarily large finite structures
(see Example~\ref{cerp.ex.FDRT-ch} and Lemma~\ref{dthms.lem.1}). This is clearly far from satisfactory.

Therefore, throughout the paper we require that each morphism
$f : (A, \Theta^\calA, \Boxed{<^\calA}) \to (B, \Theta^\calB, \Boxed{<^\calB})$ under consideration
be at least a surjective homomorphism $f : (A, \Theta^\calA) \to (B, \Theta^\calB)$ and at the same time
a rigid surjection $f : (A, \Boxed{<^\calA}) \to (B, \Boxed{<^\calB})$.
We treat this minimum set of requirements in Section~\ref{dthms.sec.homs-are-easy} where it turns out
that this setting, too, is unsatisfactory.

When it comes to surjective homomorphisms, quotient maps are usually seen as the more appropriate notion of structure maps.
Going back to our motivating example presented in Theorem~\ref{dthms.thm.met} let us note that
non-expansive rigid surjections between linearly ordered metric spaces are not only rigid homomorphisms,
but also special quotient maps (due to the special nature of metric spaces where each pair of points is in a relation).
Our main results shall be, therefore, spelled in the context where each morphism
$f : (A, \Theta^\calA, \Boxed{<^\calA}) \to (B, \Theta^\calB, \Boxed{<^\calB})$ under consideration
is a quotient map $f : (A, \Theta^\calA) \to (B, \Theta^\calB)$ and at the same time
a rigid surjection $f : (A, \Boxed{<^\calA}) \to (B, \Boxed{<^\calB})$.
As we shall see in Section~\ref{dthms.sec.qmaps} and later, this class of morphisms
turns out to be too rich for our methods, so we shall have to specialize further.
The question whether categories of linearly
ordered structures and all quotient maps which are at the same time rigid surjections have the
dual Ramsey property is an open problem.

\section{Homomorphisms are easy}
\label{dthms.sec.homs-are-easy}

In this section we prove dual Ramsey theorems for various categories of
structures and rigid homomorphisms.
A \emph{linearly ordered reflexive binary relational structure} is a triple $(A, \rho, \Boxed\sqsubset)$
where $\rho$ is a reflexive binary relation on $A$.
An \emph{empty linearly ordered reflexive binary relational structure on $n$ vertices} is a linearly ordered
reflexive binary relational structure
$$
  \calE_n = (\{1, 2, \ldots, n\}, \Delta_n, \Boxed<)
$$
where $<$ is the usual ordering of the integers and $\Delta_n = \{(i,i) : 1 \le i \le n\}$.

A mapping $f : A \to A'$ is a \emph{rigid surjective homomorphism} from $(A, \rho, \Boxed\sqsubset)$ to
$(A', \rho', \Boxed{\sqsubset'})$ if
\begin{itemize}
\item
  $f : (A, \rho) \to (A', \rho')$ is a homomorphism, that is, $(f(x), f(y)) \in \rho'$ whenever $(x, y) \in \rho$, and
\item
  $f : (A, \Boxed\sqsubset) \to (A', \Boxed{\sqsubset'})$ is a rigid surjection.
\end{itemize}
Let $\RBinRsh$ be the category whose objects are finite
linearly ordered reflexive binary relational structures
and morphisms are rigid surjective homomorphisms.

\begin{PROP}\label{dthms.prop.homs}
  Let $\CC$ be a full subcategory of $\RBinRsh$ such that $\Ob(\CC)$ contains an infinite subsequence of
  $\calE_1, \calE_2, \calE_3, \ldots, \calE_n, \ldots$.
  Then $\CC$ has the dual Ramsey property.
\end{PROP}
\begin{proof}
  The proof is analogous to the proof of Theorem~\ref{dthms.thm.met}: just use $\calE_n$ instead of $\calM_{n, \delta}$.
\end{proof}

\begin{EX}
  The following categories have the dual Ramsey property:
  \begin{itemize}
  \item
    the category whose objects are finite linearly ordered reflexive graphs and morphisms are rigid surjective homomorphisms
    (a \emph{linearly ordered reflexive graph} is a linearly ordered reflexive binary relational structure
    $(A, \rho, \Boxed\sqsubset)$ where $\rho$ is a symmetric binary relation);
  \item
    the category whose objects are finite linearly ordered reflexive oriented graphs and morphisms are rigid surjective homomorphisms
    (a \emph{linearly ordered reflexive oriented graph} is a linearly ordered reflexive binary relational structure
    $(A, \rho, \Boxed\sqsubset)$ where $\rho$ is an antisymmetric binary relation);
  \item
    the category whose objects are finite linearly ordered reflexive acyclic digraphs and morphisms are rigid surjective homomorphisms
    (a \emph{linearly ordered reflexive acyclic digraph} is a linearly ordered reflexive binary relational structure
    $(A, \rho, \Boxed\sqsubset)$ where $\rho$ has no cycles of length~$\ge$~2);
  \item
    the category whose objects are finite reflexive digraphs with linear extensions
    and morphisms are rigid surjective homomorphisms
    (a \emph{reflexive digraph with a linear extension} is a linearly ordered reflexive digraph
    $(A, \rho, \Boxed\sqsubset)$ such that $(x, y) \in \rho \Rightarrow x \sqsubset y$ whenever $x \ne y$);
  \item
    the category whose objects are finite linearly ordered posets and morphisms are rigid surjective homomorphisms
    (a \emph{linearly ordered poset} is a linearly ordered reflexive binary relational structure $(A, \rho, \Boxed\sqsubset)$
    where $\rho$ is an antisymmetric and transitive binary relation);
  \item
    the category whose objects are finite posets with linear extensions and morphisms are rigid surjective homomorphisms
    (a \emph{poset with a linear extension} is a linearly ordered reflexive binary relational structure $(A, \rho, \Boxed\sqsubset)$
    where $\rho$ is an antisymmetric and transitive binary relation and $(x, y) \in \rho \Rightarrow x \sqsubset y$
    whenever $x \ne y$).
  \end{itemize}
\end{EX}

Proposition~\ref{dthms.prop.homs} generalizes further to arbitrary finite relational structures.
An \emph{empty linearly ordered $\Theta$-structure on $n$ vertices} is a linearly ordered $\Theta$-structure
$$
  \calF_n = (\{1, 2, \ldots, n\}, \Theta^{\calF_n}, \Boxed<)
$$
where $\theta^{\calF_n} = \0$ for all $\theta \in \Theta$
and $<$ is the usual ordering of the integers.

A mapping $f : A \to B$ is a \emph{rigid surjective homomorphism} from $\calA = (A, \Theta^\calA, \Boxed{\sqsubset^\calA})$ to
$\calB = (B, \Theta^\calB, \Boxed{\sqsubset^\calB})$ if
\begin{itemize}
\item
  $f : (A, \Theta^\calA) \to (B, \Theta^\calB)$ is a homomorphism, and
\item
  $f : (A, \Boxed{\sqsubset^\calA}) \to (B, \Boxed{\sqsubset^\calB})$ is a rigid surjection.
\end{itemize}
Let $\REL_{rsh}(\Theta, \Boxed\sqsubset)$ be the category whose objects are finite
linearly ordered $\Theta$-structures and morphisms are rigid surjective homomorphisms.

\begin{PROP}
  Let $\Theta$ be a relational language and let $\Boxed\sqsubset \notin \Theta$ be a binary relational symbol.
  Let $\CC$ be a full subcategory of $\REL_{rsh}(\Theta, \Boxed\sqsubset)$ such that $\Ob(\CC)$ contains an infinite subsequence of
  $\calF_1, \calF_2, \calF_3, \ldots, \calF_n, \ldots$.
  Then $\CC$ has the dual Ramsey property. In particular, $\REL_{rsh}(\Theta, \Boxed\sqsubset)$ has the dual Ramsey
  property.
\end{PROP}
\begin{proof}
  The proof is analogous to the proof of Theorem~\ref{dthms.thm.met}: just use $\calF_n$ instead of $\calM_{n, \delta}$.
\end{proof}

\section{Dual Ramsey theorems for structures and special quotient maps}
\label{dthms.sec.qmaps}

In this section we turn to the main goal of the paper, which is to prove dual Ramsey theorems for
various categories of structures and special quotient maps. We first prove
our main technical result (Theorem~\ref{dthms.thm.erst}),
and as a consequence derive dual Ramsey theorems for acyclic digraphs with
linear extensions, posets with linear extensions, linearly ordered uniform hypergraphs and linearly ordered graphs.

For an integer $r \ge 2$, a \emph{reflexive $r$-structure with a linear extension} (or \emph{$r$-erst} for short) is a
linearly ordered reflexive structure $\calA = (A, \rho, \Boxed{\sqsubset})$ such that
$\sqsubset$ is a \emph{linear extension} of $\rho$ in the sense that
$$
  \text{if } (a_1, a_2, \ldots, a_r) \in \rho  \setminus \Delta_{A, r} \text{ then }
  a_1 \sqsubset a_2 \sqsubset \ldots \sqsubset a_r,
$$
where $r = \arity(\rho)$.

\begin{DEF}
  Let $\calA = (A, \rho, \Boxed{\sqsubset})$ and $\calA' = (A', \rho', \Boxed{\sqsubset'})$ be
  two linearly ordered relational structures where $\arity(\rho) = \arity(\rho') = r$.
  Then each homomorphism $f : (A, \rho) \to (A', \rho')$
  induces a mapping $\hat f : \rho \to \rho'$ by:
  $
    \hat f(a_1, a_2, \ldots, a_r) = (f(a_1), f(a_2), \ldots, f(a_r))
  $.
  A homomorphism $f : (A, \rho) \to (A', \rho')$ is a \emph{strong rigid quotient map}
  from $\calA$ to $\calA'$ if
  $\hat f : (\rho, \Boxed\sqsal) \to (\rho', \Boxed\sqprimesal)$ is a rigid surjection.
\end{DEF}

The following lemma justifies the name for these morphisms: it shows that
a strong rigid quotient map is a rigid surjection and a quotient map. The converse is not true, see
Fig.~\ref{dthms.fig.srqm}.

\begin{figure}
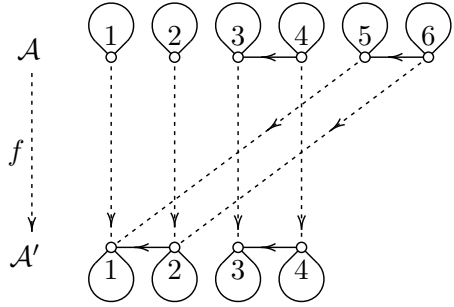

  \centering
  \input srqm3.pgf
  \caption{A rigid surjection and a quotient map which is not strong}
  \label{dthms.fig.srqm}
\end{figure}

\begin{LEM}\label{dthms.lem.strong}
  Let $\calA = (A, \rho, \Boxed\sqsubset)$ and $\calA' = (A', \rho', \Boxed{\sqsubset'})$ be
  two linearly ordered reflexive relational structures where $\arity(\rho) = \arity(\rho') = r \ge 2$,
  and let $f : (A, \rho) \to (A', \rho')$ be a homomorphism.

  $(a)$ For any $u \in A'$, if $\hat f^{-1}(u, u, \ldots, u) \ne \0$ then
  $\min \hat f^{-1}(u, u, \ldots, u) = (x, x, \ldots, x)$, where $x = \min f^{-1}(u)$.

  $(b)$ Assume that $\hat f : (\rho, \Boxed\sqsal) \to (\rho', \Boxed{\sqprimesal})$ is a rigid surjection.
  Then $f$ is a rigid surjection $(A, \Boxed\sqsubset) \to (A', \Boxed{\sqsubset'})$
  and a quotient map $(A, \rho) \to (A', \rho')$.
\end{LEM}
\begin{proof}
  $(a)$
  Assume that $\hat f^{-1}(u, u, \ldots, u) \ne \0$ and let
  $$
    \min \hat f^{-1}(u, u, \ldots, u) = (x_1, x_2, \ldots, x_r)
  $$
  where
  $|\{x_1, x_2, \ldots, x_r\}| \ge 2$. Then $\hat f(x_1, x_2, \ldots, x_r) = (u, u, \ldots, u)$.
  So, $f(x_1) = u$, whence $\hat f(x_1, x_1, \ldots, x_1) = (u, u, \ldots, u)$.
  But $(x_1, x_1, \ldots, x_1) \sqsal (x_1, x_2, \ldots, x_r)$ because
  $\tp(x_1, x_1, \ldots, x_1) \triangleleft \tp(x_1, x_2, \ldots, x_r)$.
  This contradicts the fact that $\min \hat f^{-1}(u, u, \ldots, u) = (x_1, x_2, \ldots, x_r)$.

  Thus, $\min \hat f^{-1}(u, u, \ldots, u) = (x, x, \ldots, x)$ for some $x \in A$.
  Then $f(x) = u$ whence $\min f^{-1}(u) \sqsubseteq x$.
  Assume that $\min f^{-1}(u) = t \sqsubset x$. Then $(t, t, \ldots, t) \sqsal (x, x, \ldots, x) =
  \min \hat f^{-1}(u, u, \ldots, u)$, which contradicts the fact that $(t, t, \ldots, t) \in
  \hat f^{-1}(u, u, \ldots, u)$. Therefore, $\min f^{-1}(u) = x$.

  $(b)$
  Let us start by showing that $f$ is surjective. Take any $u \in A'$. Then $(u, u, \ldots, u) \in \rho'$
  because $\rho'$ is reflexive,
  so there is an $(x_1, x_2, \ldots, x_r) \in \rho$ such that $\hat f(x_1, x_2, \ldots, x_r) =
  (u, u, \ldots, u)$ because $\hat f$ is surjective. But then $f(x_1) = u$.

  Since $f$ is a homomorphism and $\hat f : \rho \to \rho'$ is a surjective map it immediately follows that
  $f$ is a quotient map.
  
  Finally, let us prove that $f$ is a rigid surjection. Take any $u, v \in A'$ such that $u \mathrel{\sqsubset'} v$
  and let us show that $\min f^{-1}(u) \sqsubset \min f^{-1}(v)$. From $u \mathrel{\sqsubset'} v$
  it follows that $(u, u, \ldots, u) \mathrel{\sqprimesal} (v, v, \ldots, v)$, whence
  $\min \hat f^{-1}(u, u, \ldots, u) \sqsal \min \hat f^{-1}(v, v, \ldots, v)$ because $\hat f$ is a rigid surjection.
  The conclusion now follows from~$(a)$.
\end{proof}

For $r \ge 2$, let $\RslSrqm(r)$ be the category whose objects are finite $r$-erst's
and whose morphisms are strong rigid quotient maps. Our goal in this section is to prove that
$\RslSrqm(r)$ has the dual Ramsey property for every $r \ge 2$. In order to do so, we shall use the main idea
of~\cite{masul-preadj}. A pair of \emph{maps}
$$
    F : \Ob(\DD) \rightleftarrows \Ob(\CC) : G
$$
is a \emph{pre-adjunction between the categories $\CC$ and $\DD$} provided there is a family of maps
$$
    \Phi_{Y,X} : \hom_\CC(F(Y), X) \to \hom_\DD(Y, G(X))
$$
indexed by the family $\{(Y, X) \in \Ob(\DD) \times \Ob(\CC) : \hom_\CC(F(Y), X) \ne \0\}$
and satisfying the following:
\begin{itemize}
\item[(PA)]
  for every $C \in \Ob(\CC)$, every $D, E \in \Ob(\DD)$,
  every $u \in \hom_\CC(F(D), C)$ and every $f \in \hom_\DD(E, D)$ there is a $v \in \hom_\CC(F(E), F(D))$
  satisfying $\Phi_{D, C}(u) \cdot f = \Phi_{E, C}(u \cdot v)$.
  $$
    \XYMATRIX{
      F(D) \ar[rr]^u                     & & C & & D \ar[rr]^{\Phi_{D, C}(u)}                    & & G(C) \\
      F(E) \ar[u]^v \ar[urr]_{u \cdot v} & &       & & E \ar[u]^f \ar[urr]_{\Phi_{E, C}(u \cdot v)}
    }
  $$
\end{itemize}
(Note that in a pre-adjunction $F$ and $G$ are \emph{not} required to be functors, just maps from the class of objects of one of the two
categories into the class of objects of the other category; also $\Phi$ is just a family of
maps between hom-sets satisfying the requirement above.)

\begin{THM} \cite{masul-preadj} \label{opos.thm.main}
  Let $\CC$ and $\DD$ be categories such that $\CC$ has the Ramsey property and
  there is a pre-adjunction $F : \Ob(\DD) \rightleftarrows \Ob(\CC) : G$.
  Then $\DD$ has the Ramsey property.
\end{THM}

For a finite chain $\calL = \{\ell_1 \sqsubset \ell_2 \sqsubset \ldots \sqsubset \ell_N\}$ let
$$
  r \otimes \calL = (\{1, 2, \ldots, r\} \times \{\ell_1, \ell_2, \ldots, \ell_N\}, \rho_{r \otimes \calL}, \Boxed{\prec})
$$
denote an $r$-erst on $r \cdot N$ elements (written $i\ell$ instead of $(i, \ell)$) where
$$
  \rho_{r \otimes \calL} = \Delta_{\{1, 2, \ldots, r\} \times \{\ell_1, \ell_2, \ldots, \ell_N\}, r} \union
  \big\{ (1\ell, 2\ell, \ldots, r\ell) : \ell \in \calL\},
$$
and $\prec$ is the anti-lexicographic ordering of $\{1, 2, \ldots, r\} \times \{\ell_1, \ell_2, \ldots, \ell_N\}$ induced by the
respective linear orderings: $ik \prec j\ell$ iff $k \sqsubset \ell$, or $k = \ell$ and $i < j$.

\begin{THM}\label{dthms.thm.erst}
  Let $r \ge 2$ be an integer. Let $\CC$ be a full subcategory of $\RslSrqm(r)$ such that for every finite
  chain $\calL$ there is an object in $\Ob(\CC)$ isomorphic to~$r \otimes \calL$.
  Then $\CC$ has the dual Ramsey property. In particular, $\RslSrqm(r)$ has the dual Ramsey property.
\end{THM}
\begin{proof}
  Without loss of generality we may assume that $r \otimes \calL \in \Ob(\CC)$ for every finite chain~$\calL$.
  In order to prove the theorem we are going to show that there is a pre-adjunction
  $$
    F : \Ob(\CC^\op) \rightleftarrows \Ob(\CHrs^\op) : G.
  $$
  The result then follows from Theorem~\ref{opos.thm.main} and the fact that the category
  $\CHrs^\op$ has the Ramsey property (Example~\ref{cerp.ex.FDRT-ch}). Explicitly, unpacking the definition
  of pre-adjunction in case of opposite categories, we have to show the following:
  \begin{itemize}
  \item[(PA)]
    for every finite chain $\calL \in \Ob(\CHrs)$,
    all $\calA, \calB \in \Ob(\CC)$,
    for every $u \in \hom_{\CHrs}(\calL, F(\calA))$
    and every $f \in \hom_{\CC}(\calA, \calB)$
    there is a $v \in \hom_{\CHrs}(F(\calA), F(\calB))$
    satisfying $f \circ \Phi_{\calA, \calL}(u) = \Phi_{\calB, \calL}(v \circ u)$.
    $$
      \XYMATRIX{
        F(\calA) \ar@{<-}[rr]^u                          & & \calL & & \calA \ar@{<-}[rr]^{\Phi_{\calA, \calL}(u)}                    & & G(\calL) \\
        F(\calB) \ar@{<-}[u]^v \ar@{<-}[urr]_{v \circ u} & &       & & \calB \ar@{<-}[u]^f \ar@{<-}[urr]_{\Phi_{\calB, \calL}(v \circ u)}
      }
    $$
  \end{itemize}

  Take a finite chain $\calL$ and a finite $r$-erst $\calA$.
  Without loss of generality we can assume that $\calL = \{1 < 2 < \ldots < N\}$
  and $\calA = (\{1, 2, \ldots, n\}, \rho_\calA, \Boxed<)$,
  where $<$ is the usual ordering of the integers and $\rho_\calA = \{e_1 \lsal e_2 \lsal \ldots \lsal e_{q(\calA)}\}$.
  For each $i$ let $e_i = (p^1_i, p^2_i, \ldots, p^r_i)$, where, as stipulated by the definition of $r$-erst,
  either $p^1_i = p^2_i = \ldots = p^r_i$ or $p^1_i < p^2_i < \ldots < p^r_i$.
  
  Define $F$ and $G$ by $F(\calA) = (\rho_\calA, \Boxed\lsal)$ and $G(\calL) = r \otimes \calL$.
  Next, let us define
  $
    \Phi_{\calA, \calL} : \hom_{\CHrs}(\calL, F(\calA)) \to \hom_{\CC}(G(\calL), \calA)
  $.
  For a rigid surjection
  \begin{equation}\label{dthms.eq.u}
    u : \{1 < 2 < \ldots < N\} \to \{e_1 \lsal e_2 \lsal \ldots \lsal e_{q(\calA)}\}
  \end{equation}
  define $\phi_u : r \otimes \calL \to \calA$ by $\phi_u(is) = \pi_i \circ u(s)$ and put $\Phi_{\calA, \calL}(u) = \phi_u$.
  Here, $\pi_i$ stands for the $i$th projection: $\pi_i(x_1, x_2, \ldots, x_r) = x_i$.

  To show that the definition of $\Phi$ is correct we have to show that for every rigid surjection
  $u$ as in~\eqref{dthms.eq.u} the mapping  $\phi_u$
  is a strong rigid quotient map $r \otimes \calL \to \calA$.

  Let us first show that $\phi_u : (\{1, 2, \ldots, r\} \times \{1, 2, \ldots, N\}, \rho_{r \otimes \calL}) \to (A, \rho_\calA)$
  is a homomorphism. Take any $(x_1, x_2, \ldots, x_r) \in \rho_{r \otimes \calL}$. The case
  $x_1 = \ldots = x_r$ is trivial, so let us consider the case where $(x_1, x_2, \ldots, x_r) = (1s, 2s, \ldots, rs)$
  for some $s \in \calL$:
  \begin{align*}
    (\phi_u(x_1), \phi_u(x_2), \ldots, \phi_u(x_r))
      &= (\phi_u(1s), \phi_u(2s), \ldots, \phi_u(rs))\\
      &= (\pi_1 \circ u(s), \pi_2 \circ u(s), \ldots, \pi_r \circ u(s))\\
      &= u(s) \in \rho_\calA.
  \end{align*}

  Let us now show that $\hat\phi_u : (\rho_{r \otimes \calL}, \Boxed\precsal) \to (\rho_\calA, \Boxed\lsal)$
  is a rigid surjection. Note that
  $
    \hat \phi_u\big((i_1s, i_2s, \ldots, i_rs)\big) = (\pi_{i_1} \circ u(s), \pi_{i_2} \circ u(s), \ldots, \pi_{i_r} \circ u(s)).
  $
  
  \medskip

  Claim~1. Take any $e_k = (p^1_k, p^2_k, \ldots, p^r_k) \in \rho_\calA$, $1 \le k \le F(\calA)$, and let
  $t = \min u^{-1}(e_k)$. Then the following holds:
  \begin{itemize}
  \item[]
    $1^\circ$ if $p^1_k < p^2_k < \ldots < p^r_k$ then $\min \hat \phi_u^{-1}(e_k) = (1t, 2t, \ldots, rt)$;
  \item[]
    $2^\circ$ if $p^1_k = p^2_k = \ldots = p^r_k$ then $\min \hat \phi_u^{-1}(e_k) = (1t, 1t, \ldots, 1t)$.
  \end{itemize}

  \medskip

  Proof of $1^\circ$. Assume that $p^1_k < p^2_k < \ldots < p^r_k$. From
  $$
    \hat \phi_u\big((1t, 2t, \ldots, rt)\big) = u(t) = e_k
  $$
  it follows that $(1t, 2t, \ldots, rt) \in \hat \phi_u^{-1}(e_k)$, so it suffices to show that
  $$
    (1t, 2t, \ldots, rt) \precsaleq (i_1s, i_2s, \ldots, i_rs)
  $$
  for every $(i_1s, i_2s, \ldots, i_rs) \in \hat \phi_u^{-1}(e_k)$.
  Take any $(i_1s, i_2s, \ldots, i_rs) \in \rho_{r \otimes \calL}$ such that
  $\hat \phi_u\big((i_1s, i_2s, \ldots, i_rs)\big)  = e_k$. Then
  \begin{equation}\label{dthms.eq.1}
    (\pi_{i_1} \circ u(s), \pi_{i_2} \circ u(s), \ldots, \pi_{i_r} \circ u(s)) = (p^1_k, p^2_k, \ldots, p^r_k).
  \end{equation}
  Let $u(s) = e_m$. Since $p^1_k < p^2_k < \ldots < p^r_k$, the only possibility to achieve~\eqref{dthms.eq.1}
  is $i_1 = 1$, $i_2 = 2$, \ldots, $i_r = r$ and $m = k$.
  Therefore, $u(s) = e_k$ whence $s \in u^{-1}(e_k)$, so $t \le s$. Now
  $$
    (1t, 2t, \ldots, rt) \precsaleq (1s, 2s, \ldots, rs) = (i_1s, i_2s, \ldots, i_rs).
  $$
  
  \medskip

  Proof of $2^\circ$. Assume that $p^1_k = p^2_k = \ldots = p^r_k$. From
  $$
    \hat \phi_u\big((1t, 1t, \ldots, 1t)\big) = (p^1_k, p^1_k, \ldots, p^1_k) = e_k
  $$
  it follows that $(1t, 1t, \ldots, 1t) \in \hat \phi_u^{-1}(e_k)$, so it suffices to show that
  $$
    (1t, 1t, \ldots, 1t) \precsaleq (i_1s, i_2s, \ldots, i_rs)
  $$
  for every $(i_1s, i_2s, \ldots, i_rs) \in \hat \phi_u^{-1}(e_k)$.
  Take any $(i_1s, i_2s, \ldots, i_rs) \in \rho_{r \otimes \calL}$ such that
  $\hat \phi_u\big((i_1s, i_2s, \ldots, i_rs)\big)  = e_k$. Then
  \begin{equation}\label{dthms.eq.2}
    (\pi_{i_1} \circ u(s), \pi_{i_2} \circ u(s), \ldots, \pi_{i_r} \circ u(s)) = (p^1_k, p^1_k, \ldots, p^1_k),
  \end{equation}
  whence
  \begin{equation}\label{dthms.eq.pi-pi}
    \pi_{i_\alpha} \circ u(s) = \pi_{i_\beta} \circ u(s) \text{ for all $\alpha$ and $\beta$}.
  \end{equation}
  Let $u(s) = e_m = (p^1_m, p^2_m, \ldots, p^r_m)$. 
  
  \medskip
  
  $2.1^\circ$. Assume that $p^1_m = p^2_m = \ldots = p^r_m$. Then
  $$
    (p^1_m, p^1_m, \ldots, p^1_m) = (\pi_{i_1} \circ u(s), \pi_{i_2} \circ u(s), \ldots, \pi_{i_r} \circ u(s)) = (p^1_k, p^1_k, \ldots, p^1_k)
  $$
  whence $u(s) = e_m = e_k$. Then $s \in u^{-1}(e_k)$, so $t \le s$. Hence,
  $$
    (1t, 1t, \ldots, 1t) \precsaleq (i_1s, i_2s, \ldots, i_rs)
  $$
  for any choice $(i_1, i_2, \ldots, i_r) \in
  \{(1,1, \ldots, 1), \ldots, (r, r, \ldots, r), (1, 2, \ldots, r)\}$.
  
  \medskip
  
  $2.2^\circ$. Assume, now, that $p^1_m < p^2_m < \ldots < p^r_m$.
  Then \eqref{dthms.eq.pi-pi} implies that $i_1 = i_2 = \ldots = i_r$. Let $i_1 = i_2 = \ldots = i_r = \alpha$.
  Since $\tp(p^1_k, p^1_k, \ldots, p^1_k) \triangleleft \tp(p^1_m, p^2_m, \ldots, p^r_m)$
  it follows that
  $$
    e_k = (p^1_k, p^1_k, \ldots, p^1_k) \precsal (p^1_m, p^2_m, \ldots, p^r_m) = e_m.
  $$
  Therefore, $t = \min u^{-1}(e_k) < \min u^{-1}(e_m) \le s$ whence
  $(1t, 1t, \ldots, 1t) \precsal (\alpha s, \alpha s, \ldots, \alpha s) = (i_1s, i_2s, \ldots, i_rs)$.

  \medskip
  
  This concludes the proof of Claim~1.
  
  \medskip

  We are now ready to show that $\hat\phi_u : (\rho_{r \otimes \calL}, \Boxed\precsal) \to (\rho_\calA, \Boxed\lsal)$
  is a rigid surjection.
  Take any $e_i = (p^1_i, p^2_i, \ldots, p^r_i), e_j  = (p^1_j, p^2_j, \ldots, p^r_j) \in \rho_\calA$
  such that $e_i \lsal e_j$. Then $\min u^{-1}(e_i) < \min u^{-1}(e_j)$
  because $u$ is a rigid surjection. Let $s = \min u^{-1}(e_i)$ and $t = \min u^{-1}(e_j)$.
  \begin{itemize}
  \item
    If $p^1_i = p^2_i = \ldots = p^r_i$ and $p^1_j = p^2_j = \ldots = p^r_j$ then, by Claim~1,
    $\min \hat \phi_u^{-1}(e_i) = (1s, 1s, \ldots, 1s) \precsal (1t, 1t, \ldots, 1t) = \min \hat \phi_u^{-1}(e_j)$.
  \item
    If $p^1_i = p^2_i = \ldots = p^r_i$ and $p^1_j < p^2_j < \ldots < p^r_j$ then, by Claim~1,
    $\min \hat \phi_u^{-1}(e_i) = (1s, 1s, \ldots, 1s) \precsal (1t, 2t, \ldots, rt) = \min \hat \phi_u^{-1}(e_j)$.
  \item
    If $p^1_i < p^2_i < \ldots < p^r_i$ and $p^1_j < p^2_j < \ldots < p^r_j$ then, by Claim~1,
    $\min \hat \phi_u^{-1}(e_i) = (1s, 2s, \ldots, rs) \precsal (1t, 2t, \ldots, rt) = \min \hat \phi_u^{-1}(e_j)$.
  \end{itemize}
  This proves that $\hat \phi_u$ is a rigid surjection and the definition of $\Phi$ is correct.

  We still have to show that this family of maps satisfies the requirement~(PA). But this is easy.
  Let $\calB = (\{1, 2, \ldots, \ell\}, \rho_\calB, \Boxed<)$ be a finite $r$-erst
  where $<$ is the usual ordering of the integers, and let $f : \calA \to \calB$ be a strong rigid quotient map.
  Then $\hat f : (\rho_\calA, \Boxed\lsal) \to (\rho_\calB, \Boxed\lsal)$
  is a rigid surjection by definition. Let us show that $f \circ \phi_u = \phi_{\hat f \circ u}$:
  $$
    f \circ \phi_u(is) = f \circ \pi_i \circ u(s) = \pi_i \circ \hat f \circ u(s) = \phi_{\hat f \circ u}(is).
  $$
  This calculation relies on the fact that $f \circ \pi_i = \pi_i \circ \hat f$ which is clearly true:
  $f \circ \pi_i(x_1, x_2, \ldots, x_r) = f(x_i) = \pi_i(f(x_1), f(x_2), \ldots, f(x_r)) = \pi_i \circ \hat f (x_1, x_2, \ldots, x_r)$.
\end{proof}

Specializing the above result for $r = 2$ we get the following. 

\begin{COR}
  The following categories have the dual Ramsey property:
  \begin{itemize}
  \item
    the category $\RdlSrqm$ whose objects are finite reflexive digraphs with linear extensions
    and morphisms are strong rigid quotient maps;
  \item
    the category $\EPosSrqm$ whose objects are finite posets with linear extensions and morphisms are strong rigid quotient maps.
  \end{itemize}
\end{COR}
\begin{proof}
  For the first item it suffices to note that $\RdlSrqm = \RslSrqm(2)$.
  For the second item it suffices to note that $\EPosSrqm$ is a full subcategory of $\RdlSrqm$ such that $2 \otimes \calL \in
  \Ob(\EPosSrqm)$ for every finite chain~$\calL$.
\end{proof}

As another corollary of Theorem~\ref{dthms.thm.erst}
we shall now prove a dual Ramsey theorem for reflexive graphs and hypergraphs together with special quotient maps.

\begin{DEF}
  For a chain $\calA = (A, \Boxed\sqsubset)$ let us define $\sqsal$ on $\calP(A)$ as follows
  (``sal'' in the subscript stands for ``special anti-lexicographic''). Take any $X, Y \in \calP(A)$.
  \begin{itemize}
  \item
    $\0$ is the least element of $\calP(A)$ with respect to~$\sqsal$;
  \item
    if $X = \{x\}$ and $Y = \{y\}$ then $X \sqsal Y$ iff $x \sqsubset y$;
  \item
    if $|X| = 1$ and $|Y| > 1$ then $X \sqsal Y$;
  \item
    if $|X| > 1$ and $|Y| > 1$ then $X \sqsal Y$ iff $X \sqalex Y$.
  \end{itemize}
\end{DEF}

For a set $A$ and an integer $k$ let $\binom Ak$ denote the set of all the $k$-element subsets of~$A$.
Let $r \ge 2$ be an integer. A \emph{linearly ordered reflexive $r$-uniform hypergraph}
is a triple $(V, E, \Boxed\sqsubset)$ where $E = \binom V1 \union S$ for some $S \subseteq \binom Vr$,
and $\sqsubset$ is a linear order on~$V$.
A mapping $f : V \to V'$ between (unordered) reflexive $r$-uniform hypergraphs
$\calG = (V, E)$ and $\calG' = (V', E')$ is a \emph{hypergraph homomorphism} 
if $e \in E$ implies $\{f(x) : x \in e\} \in E'$.

\begin{DEF}
  Let $\calG = (V, E, \Boxed\sqsubset)$ and $\calG' = (V', E', \Boxed{\sqsubset'})$ be linearly ordered reflexive $r$-uniform
  hypergraphs. Each hypergraph homomorphism $f : (V, E) \to (V', E')$ induces a mapping $\tilde f : E \to E'$ straightforwardly:
  $\tilde f(e) = \{f(x) : x \in e\}$. A hypergraph homomorphism $f : (V, E) \to (V', E')$ is a \emph{strong rigid quotient map
  of hypergraphs} if
  \begin{itemize}
  \item
    $\tilde f : (E, \Boxed\sqsal) \to (E', \Boxed{\sqprimesal})$ is a rigid surjection, and
  \item
    for every $e = \{x_1, \ldots, x_r\} \in E$, if $\restr fe$ is not a constant map then
    $x_i \sqsubset x_j \Rightarrow f(x_i) \mathrel{\sqsubset'} f(x_j)$ for all $i$ and $j$.
  \end{itemize}
\end{DEF}

\begin{EX}
  Let $C_3 = (\{1, 2, 3\}, E_3, \Boxed<)$ and $C_4 = (\{1, 2, 3, 4\}, E_4, \Boxed<)$ be the
  reflexive 3-cycle and the reflexive 4-cycle, respectively, where
  $E_3 = \{1, 2, 3, 12, 23, 31\}$, $E_4 = \{1, 2, 3, 4, 12, 23, 34, 14\}$,
  and $<$ is the usual ordering of the integers. Let
  $$
    f = \begin{pmatrix}
      1 & 2 & 3 & 4\\
      1 & 1 & 2 & 3
    \end{pmatrix}
    \text{,\quad}
    g = \begin{pmatrix}
      1 & 2 & 3 & 4\\
      1 & 2 & 3 & 3
    \end{pmatrix}
  $$
  be two quotient maps $(\{1, 2, 3, 4\}, E_4) \to (\{1, 2, 3\}, E_3)$. Then $f$ is a strong rigid quotient map and $g$
  is not. Namely,
  $$
    \tilde f = \begin{pmatrix}
      1 & 2 & 3 & 4 & 12 & 23 & 14 & 34\\
      1 & 1 & 2 & 3 & 1  & 12 & 13 & 23
    \end{pmatrix}
    \text{,\quad}
    \tilde g = \begin{pmatrix}
      1 & 2 & 3 & 4 & 12 & 23 & 14 & 34\\
      1 & 2 & 3 & 3 & 12 & 23 & 13 & 3
    \end{pmatrix}
  $$
  and we can now easily see that $\tilde f : (E_4, \lsal) \to (E_3, \lsal)$ is rigid while
  $\tilde g : (E_4, \lsal) \to (E_3, \lsal)$ is not.
\end{EX}

Let $\ERhgSrqm(r)$ be the category whose objects are finite linearly ordered reflexive $r$-uniform hypergraphs
and whose morphisms are strong rigid quotient maps of hypergraphs.
Note that $\ERhgSrqm(2)$ is the category whose objects are finite linearly ordered reflexive graphs
and whose morphisms are strong rigid quotient maps of graphs. Let us denote this category by~$\GraSrqm$.

For a finite chain $\calL = \{\ell_1 \sqsubset \ell_2 \sqsubset \ldots \sqsubset \ell_N\}$ and $r \ge 2$ let
$$
  r \boxtimes \calL = (\{1, 2, \ldots, r\} \times \{\ell_1, \ell_2, \ldots, \ell_N\}, E_{r \boxtimes \calL}, \Boxed{\prec})
$$
denote a linearly ordered reflexive $r$-uniform hypergraph on $r \cdot N$ vertices (written $i\ell$ instead of $(i, \ell)$)
where
\begin{align*}
  E_{r \boxtimes \calL} = \mathstrut &\big\{ \{i\ell_j\} : 1 \le i \le r, 1 \le j \le N \big\} \union \mathstrut\\
                          & \qquad \union \big\{ \{1\ell_j, 2\ell_j, \ldots, r\ell_j\} : 1 \le j \le N \big\},
\end{align*}
and $\prec$ is the anti-lexicographic ordering of $\{1, 2, \ldots, r\} \times \{\ell_1, \ell_2, \ldots, \ell_N\}$ induced by the
respective linear orderings: $i\ell \prec jk$ iff $\ell \sqsubset k$, or $\ell = k$ and $i < j$.

\begin{COR}\label{dthms.cor.lorg}
  Let $\CC$ be a full subcategory of $\ERhgSrqm(r)$, $r \ge 2$,
  such that for every finite chain $\calL$ there is an object in $\Ob(\CC)$
  isomorphic to~$r \boxtimes \calL$. Then $\CC$ has the dual Ramsey property.
  In particular, the following categories have the dual Ramsey property:
  \begin{itemize}
  \item
    the category $\ERhgSrqm(r)$ for every $r \ge 2$;
  \item
    the category $\GraSrqm$;
  \item
    the full subcategory of $\GraSrqm$ spanned by bipartite graphs;
  \item
    the full subcategory of $\GraSrqm$ spanned by $K_n$-free graphs (where $n \ge 3$ is fixed).
  \end{itemize}
\end{COR}
\begin{proof}
  Let us start by proving that $\ERhgSrqm(r)$ and $\RslSrqm(r)$ are isomorphic.
  Define a functor
  $$
    F : \ERhgSrqm(r) \to \RslSrqm(r) : (V, E, \Boxed\sqsubset) \mapsto (V, \rho, \Boxed\sqsubset) : f \mapsto f,
  $$
  where
  \begin{align*}
    \rho = \Delta_{V, r} \union \mathstrut & \{(x_1, x_2, \ldots, x_r) \in V^r : \\
                                           & \qquad x_1 \sqsubset x_2 \sqsubset \ldots \sqsubset x_r
                                                                   \text{ and } \{x_1, x_2, \ldots, x_r\} \in E\}.
  \end{align*}
  On the other hand, define a functor
  $$
    G : \RslSrqm(r) \to \ERhgSrqm(r) : (A, \rho, \Boxed\sqsubset) \mapsto (A, E, \Boxed\sqsubset) : f \mapsto f,
  $$
  where
  $$
    E = \big\{\{x_1, x_2, \ldots, x_r\} : (x_1, x_2, \ldots, x_r) \in \rho\big\}.
  $$
  By construction, $F$ and $G$ are mutually inverse functors,
  so the categories $\ERhgSrqm(r)$ and $\RslSrqm(r)$ are isomorphic. However, we still have to show that the
  functors $F$ and $G$ are well defined.  
  It is easy to see that both $F$ and $G$ are well defined on objects. Let us show that both $F$ and $G$
  are well defined on morphisms.

  Let $f : (V, E, \Boxed\sqsubset) \to (V', E', \Boxed{\sqsubset'})$ be a morphism in $\ERhgSrqm(r)$ and let us show
  that $f : (V, \rho) \to (V', \rho')$ is a homomorphism.
  Take any $(x_1, x_2, \ldots, x_r) \in \rho \setminus \Delta_{V,r}$. Then
  $$
    \{x_1, x_2, \ldots, x_r\} \in E \text{ and } x_1 \sqsubset x_2 \sqsubset \ldots \sqsubset x_r.
  $$
  If $f(x_1) = f(x_2) = \ldots = f(x_r)$ we are done. Assume, therefore, that $\restr{f}{\{x_1, x_2, \ldots, x_r\}}$ is not
  a constant map. Then, by the definition of morphisms in $\ERhgSrqm(r)$ it follows that
  $$
    \{f(x_1), f(x_2), \ldots, f(x_r)\} \in E \text{ and } f(x_1) \mathrel{\sqsubset'} f(x_2) \mathrel{\sqsubset'} \ldots \mathrel{\sqsubset'} f(x_r).
  $$
  Therefore, $(f(x_1), f(x_2), \ldots, f(x_r)) \in \rho'$.
  
  Conversely,
  let $f : (V, \rho, \Boxed\sqsubset) \to (V', \rho', \Boxed{\sqsubset'})$ be a morphism in $\RslSrqm(r)$ and let us show
  that $f : (V, E) \to (V', E')$ is a homomorphism of hypergraphs.
  Take any $\{x_1, x_2, \ldots, x_r\} \in E$. If $f(x_1) = f(x_2) = \ldots = f(x_r)$ we are done. Assume, therefore,
  that $\restr{f}{\{x_1, x_2, \ldots, x_r\}}$ is not a constant map. Without loss of generality we may assume that
  $x_1 \sqsubset x_2 \sqsubset \ldots \sqsubset x_r$. Then $(x_1, x_2, \ldots, x_r) \in \rho$, so
  $(f(x_1), f(x_2), \ldots, f(x_r)) \in \rho'$ because $f$ is a morphism in $\RslSrqm(r)$. Therefore,
  $\{f(x_1), f(x_2), \ldots, f(x_r)\} \in E'$ and
  $f(x_1) \mathrel{\sqsubset'} f(x_2) \mathrel{\sqsubset'} \ldots \mathrel{\sqsubset'} f(x_r)$.
  So, $f$ is a homomorphism satisfying the additional requirement that
  for every $e = \{x_1, \ldots, x_r\} \in E$, if $\restr fe$ is not a constant map then
  $x_i \sqsubset x_j \Rightarrow f(x_i) \mathrel{\sqsubset'} f(x_j)$ for all $i$ and $j$.
  
  In order to complete the proof that $F$ and $G$ are well defined on morphisms we still have to show that
  $\hat f : (\rho, \Boxed\sqsal) \to (\rho', \Boxed\sqprimesal)$ is a rigid surjection iff
  $\tilde f : (E, \Boxed\sqsal) \to (E', \Boxed\sqprimesal)$ is a rigid surjection. But this follows straightforwardly
  from the following facts:
  \begin{itemize}
  \item
    the mapping $\xi : \rho \to E : (x_1, x_2, \ldots, x_r) \mapsto \{x_1, x_2, \ldots, x_r\}$ is an isomorphism from
    $(\rho, \Boxed\sqsal)$ to $(E, \Boxed\sqsal)$,
  \item
    the mapping $\xi' : \rho' \to E' : (x_1, x_2, \ldots, x_r) \mapsto \{x_1, x_2, \ldots, x_r\}$ is an isomorphism from
    $(\rho', \Boxed\sqprimesal)$ to $(E', \Boxed\sqprimesal)$, and
  \item
    $\tilde f = \xi' \circ \hat f \circ \xi^{-1}$.
  \end{itemize}
  Therefore, $F$ and $G$ are well defined functors. 

  The first item in the statement of the theorem now follows immediately from Theorem~\ref{dthms.thm.erst} having in mind that
  $F(r \boxtimes \calL) = r \otimes \calL$ for every chain~$\calL$.
  As for the remaining items, note that $\GraSrqm = \ERhgSrqm(2)$, and that both subcategories of $\GraSrqm$
  mentioned in the third and the fourth item contain $2 \boxtimes \calL$ for every finite chain~$\calL$.
\end{proof}

\section{The Dual Ne\v set\v ril-R\"odl Theorem}
\label{dthms.sec.DNRT}

One of the cornerstones of the structural Ramsey theory is the Ne\v set\v ril-R\"odl Theorem
(Theorem~\ref{dthms.thm.NRT}). In this section we prove a dual form of the Ne\v set\v ril-R\"odl Theorem,
albeit only in its restricted form which does not account for subclasses defined by forbidden ``quotients''.

Let $\Theta$ be a relational language and let $\Boxed\sqsubset \notin \Theta$ be a new binary relational symbol.
A \emph{reflexive $\Theta$-structure with a linear extension} (or \emph{$\Theta$-erst} for short) is a
linearly ordered reflexive $\Theta$-structure $\calA = (A, \Theta^\calA, \Boxed{\sqsubset^\calA})$ such that
$(A, \theta^\calA, \Boxed{\sqsubset^\calA})$ is an $\arity(\theta)$-erst for every $\theta \in \Theta$.

\begin{DEF}
  Let $\calA = (A, \Theta^\calA, \Boxed{\sqsubset^\calA})$ and $\calB = (B, \Theta^\calB, \Boxed{\sqsubset^\calB})$ be two
  $\Theta$-erst's. A homomorphism $f : (A, \Theta^\calA) \to (B, \Theta^\calB)$ is a \emph{strong rigid quotient map}
  from $\calA$ to $\calB$ if $\hat f : (\theta^\calA, \Boxed{\sqsalidx^\calA}) \to (\theta^\calB, \Boxed{\sqsalidx^\calB})$
  is a rigid surjection for every $\theta \in \Theta$.
\end{DEF}

Let $\RslSrqm(\Theta, \Boxed\sqsubset)$ be the category whose objects are all finite $\Theta$-erst's and whose morphisms
are strong rigid quotient maps. In order to prove that $\RslSrqm(\Theta, \Boxed\sqsubset)$ has the dual Ramsey property
we shall employ a strategy devised in~\cite{masul-drp-perm}. Let us recall two technical statements from~\cite{masul-drp-perm}.

A \emph{diagram} in a category $\CC$ is a functor $F : \Delta \to \CC$ where the category $\Delta$ is referred to as the
\emph{shape of the diagram}. We shall say that a diagram $F : \Delta \to \CC$ is \emph{consistent in $\CC$} if there exists a $C \in \Ob(\CC)$
and a family of morphisms $(h_\delta : F(\delta) \to C)_{\delta \in \Ob(\Delta)}$ such that for every
morphism $g : \delta \to \gamma$ in $\Delta$ we have $h_\gamma \cdot F(g) = h_\delta$:
$$
  \xymatrix{
     & C & \\
    F(\delta) \ar[ur]^{h_\delta} \ar[rr]_{F(g)} & & F(\gamma) \ar[ul]_{h_\gamma}
  }
$$
We say that $C$ together with the family of morphisms
$(h_\delta)_{\delta \in \Ob(\Delta)}$ forms a \emph{compatible cone in $\CC$ over the diagram~$F$}.

A \emph{binary category} is a finite, acyclic, bipartite digraph with loops 
where all the arrows go from one class of vertices into the other
and the out-degree of all the vertices in the first class is~2 (modulo loops):
$$
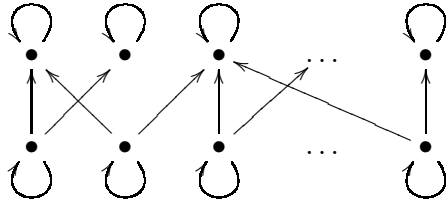
\xymatrix{
  & \\
  \bullet \ar@(ur,ul) & \bullet \ar@(ur,ul) & \bullet \ar@(ur,ul) & \ldots & \bullet \ar@(ur,ul) \\
  \bullet \ar@(dr,dl) \ar[u] \ar[ur] & \bullet \ar@(dr,dl) \ar[ur] \ar[ul] & \bullet \ar@(dr,dl) \ar[u] \ar[ur] & \ldots & \bullet \ar@(dr,dl) \ar[u] \ar[ull] \\
  &
}
$$
A \emph{binary diagram} in a category $\CC$ is a functor $F : \Delta \to \CC$ where $\Delta$ is a binary category,
$F$ takes the bottom row of $\Delta$ onto the same object, and takes the top row of $\Delta$ onto
the same object, Fig.~\ref{nrt.fig.2}.
A subcategory $\DD$ of a category $\CC$ is \emph{closed for binary diagrams} if every binary diagram
$F : \Delta \to \DD$ which is consistent in $\CC$ is also consistent in~$\DD$.

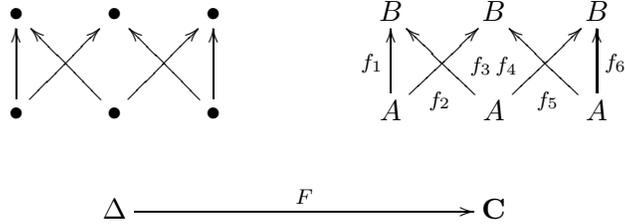
\begin{figure}
  $$
  \xymatrix{
    \bullet & \bullet & \bullet
    & & B & B & B
  \\
    \bullet \ar[u] \ar[ur] & \bullet \ar[ur] \ar[ul] & \bullet \ar[ul] \ar[u]
    & & A \ar[u]^{f_1} \ar[ur]_(0.3){f_2} & A \ar[ur]^(0.3){f_4} \ar[ul]_(0.3){f_3} & A \ar[ul]^(0.3){f_5} \ar[u]_{f_6}
  \\
    & \Delta \ar[rrrr]^F  & & & & \CC  
  }
  $$
  \caption{A binary diagram in $\CC$ (of shape $\Delta$)}
  \label{nrt.fig.2}
\end{figure}

\begin{THM}\label{nrt.thm.1} \cite{masul-drp-perm}
  Let $\CC$ be a category such that every morphism in $\CC$ is monic and
  such that $\hom_\CC(A, B)$ is finite for all $A, B \in \Ob(\CC)$, and let $\DD$ be a
  (not necessarily full) subcategory of~$\CC$.
  If $\CC$ has the Ramsey property and $\DD$ is closed for binary diagrams, then $\DD$ has the Ramsey property.
\end{THM}

We shall also need a categorical version of the Product Ramsey Theorem for Finite Structures of M.~Soki\'c~\cite{sokic2}.
We proved this statement in the categorical context in~\cite{masul-drp-perm} where we used this abstract version
to prove that the class of finite permutations has the dual Ramsey property.

\begin{THM}\label{sokic-prod} \cite{masul-drp-perm}
  Let $\CC_1$ and $\CC_2$ be categories such that $\hom_{\CC_i}(A, B)$ is finite
  for all $A, B \in \Ob(\CC_i)$, $i \in \{1, 2\}$.
  If $\CC_1$ and $\CC_2$ both have the Ramsey property then
  $\CC_1 \times \CC_2$ has the Ramsey property.
  
  Consequently, if $\CC_1, \ldots, \CC_n$ are categories with the Ramsey property then the category
  $\CC_1 \times \ldots \times \CC_n$ has the Ramsey property.
\end{THM}

We are now ready to prove the main technical result of this section.

\begin{PROP}\label{dthms.prop.theta-erst}
  The category $\RslSrqm(\Theta, \Boxed\sqsubset)$ has the dual Ramsey property
  for every relational language~$\Theta$.
\end{PROP}
\begin{proof}
  Part~I. Assume, first, that $\Theta = \{ \theta_1, \theta_2, \ldots, \theta_n \}$ is a finite relational language and let
  $r_i = \arity(\theta_i)$, $1 \le i \le n$.
  Let $\CC_i$ denote the category $\RslSrqm(\{\theta_i\}, \Boxed\sqsubset)$, $1 \le i \le n$.
  For each $i$ we have that $\RslSrqm(\{\theta_i\}, \Boxed\sqsubset) = \RslSrqm(r_i)$,
  so $\RslSrqm(\{\theta_i\}, \Boxed\sqsubset)$ has the dual Ramsey
  property (Theorem~\ref{dthms.thm.erst}).
  
  For an object $\calA = (A, \theta_1^\calA, \ldots, \theta_n^\calA, \Boxed{\sqsubset^\calA}) \in
  \Ob(\RslSrqm(\Theta, \Boxed{\sqsubset}))$ let $\calA^{(i)} = (A, \theta_i^\calA, \Boxed{\sqsubset^\calA}) \in \Ob(\CC_i)$.
  As we have just seen each $\CC_i$ has the dual Ramsey property, so the product category
  $\CC_1 \times \ldots \times \CC_n$ has the dual Ramsey property
  by Theorem~\ref{sokic-prod}. Let $\DD$ be the following subcategory of $\CC_1 \times \ldots \times \CC_n$:
  \begin{itemize}
  \item
    every $\calA = (A, \theta_1^\calA, \ldots, \theta_n^\calA, \Boxed{\sqsubset^\calA}) \in
    \Ob(\RslSrqm(\Theta, \Boxed{\sqsubset}))$ gives rise to an object
    $\overline \calA = (\calA^{(1)}, \ldots, \calA^{(n)})$ of $\DD$, and these are the only objects in~$\DD$;
  \item
    every morphism $f : \calA \to \calB$ in $\RslSrqm(\Theta, \Boxed{\sqsubset})$ gives rise to a morphism
    $\overline f = (f, \ldots, f) : \overline \calA \to \overline \calB$ in $\DD$, and these are the only morphisms in~$\DD$.
  \end{itemize}
  Clearly, the categories $\DD$ and $\RslSrqm(\Theta, \Boxed{\sqsubset})$ are isomorphic,
  so in order to complete the proof of the lemma it suffices to show that $\DD$ has the dual Ramsey property.

  As $\DD$ is a subcategory of $\CC_1 \times \ldots \times \CC_n$ and the latter one has the dual Ramsey property,
  following Theorem~\ref{nrt.thm.1} it suffices to show that $\DD^\op$ is closed for binary diagrams
  in $(\CC_1 \times \ldots \times \CC_n)^\op$.

  Let $F : \Delta \to \DD^\op$ be a binary diagram in $\DD^\op$ where the top row consists of copies of
  $\overline\calB \in \Ob(\DD)$ and the bottom row consists of copies of $\overline\calA \in \Ob(\DD)$
  for some $\calA = (A, \theta_1^\calA, \ldots, \theta_n^\calA, \Boxed{\sqsubset^\calA})$ and
  $\calB = (B, \theta_1^\calB, \ldots, \theta_n^\calB, \Boxed{\sqsubset^\calB})$.
  Assume that $F$ is consistent in $(\CC_1 \times \ldots \times \CC_n)^\op$
  and let $(\calC_1, \ldots, \calC_n)$ together with the morphisms $q_1, \ldots, q_k$ be a compatible cone
  in $(\CC_1 \times \ldots \times \CC_n)^\op$ over~$F$:
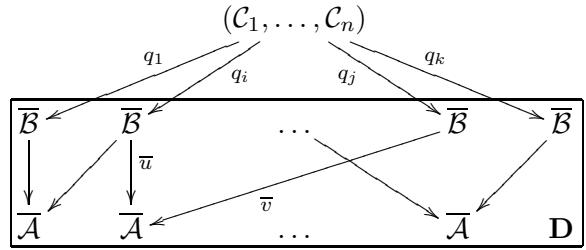
  $$
  \xymatrix{
    & & (\calC_1, \ldots, \calC_n)
  \\
    \overline \calB \ar@{<-}[urr]^{q_1} & \overline \calB \ar@{<-}[ur]_(0.6){q_i} & \ldots & \overline \calB \ar@{<-}[ul]^(0.6){q_j} & \overline \calB \ar@{<-}[ull]_{q_k}
  \\
    \overline \calA \ar@{<-}[u] \ar@{<-}[ur] & \overline \calA \ar@{<-}[urr]_(0.4){\overline v} \ar@{<-}[u]_(0.65){\overline u} & \ldots & \overline \calA \ar@{<-}[ur] \ar@{<-}[ul] & \DD
    \save "2,1"."3,5"*[F]\frm{} \restore
  }
  $$
  Let $\calC_i = (C_i, \theta_i^{\calC_i}, \Boxed{\sqsubset^i})$ and $q_i = (q_i^1, \ldots, q_i^n)$ where
  $q_i^s : \calC_s \to \calB^{(s)}$ is a strong rigid quotient map. Without loss of generality we can assume
  that $C_1, C_2, \ldots, C_n$ are pairwise disjoint sets.
  Let $\calD = (D, \theta_1^\calD, \ldots, \theta_n^\calD, \Boxed{\sqsubset^\calD})$ where
  \begin{itemize}
  \item
    $D = C_1 \union C_2 \union \ldots \union C_n$,
  \item
    $\theta_i^\calD = \Delta_{D, r_i} \union \theta_i^{\calC_i}$, $1 \le i \le n$, and
  \item
    $\sqsubset^\calD$ is the linear order on $D$ obtained by concatenating the linear orders
    $\sqsubset^1$, $\sqsubset^2$, \ldots, $\sqsubset^n$; in other words, $\sqsubset^\calD$ is the unique linear order on $D$ such that
    $\restr{\Boxed{\sqsubset^\calD}}{C_i} = \Boxed{\sqsubset^i}$, $1 \le i \le n$, and
    if $x \in C_i$ and $y \in C_j$ where $i < j$ then $x \mathrel{\sqsubset^\calD} y$.
  \end{itemize}
  Clearly, $\calD \in \Ob(\RslSrqm(\Theta, \Boxed\sqsubset))$, so $\overline \calD \in \Ob(\DD)$.

  For each morphism $q_i = (q_i^1, \ldots, q_i^n)$ let $\phi_i : D \to B$ be the following mapping:
  $$
    \phi_i(x) = \begin{cases}
      q_i^1(x), & x \in C_1\\
      q_i^2(x), & x \in C_2\\
      \vdots\\
      q_i^n(x), & x \in C_n.
    \end{cases}
  $$
  Let us show that $\phi_i : \calD \to \calB$ is a strong rigid quotient map, $1 \le i \le k$. It is easy to see that each $\phi_i$ is
  a homomorphism $(D, \theta_1^\calD, \ldots, \theta_n^\calD) \to (B, \theta_1^\calB, \ldots, \theta_n^\calB)$.
  Take any $s$ and any $(x_1, x_2, \ldots, x_{r_s}) \in \theta^\calD_s \setminus \Delta_{D, r_s}$. Then $(x_1, x_2, \ldots, x_{r_s}) \in \theta^{\calC_s}_s$ whence
  $\{x_1, x_2, \ldots, x_{r_s}\} \subseteq C_s$. But then
  $$
    (\phi_i(x_1), \phi_i(x_2), \ldots, \phi_i(x_{r_s})) = (q_i^s(x_1), q_i^s(x_2), \ldots, q_i^s(x_{r_s})) \in \theta^\calB_s
  $$
  because $q_i^s : \calC_s \to \calB^{(s)}$ is a homomorphism.

  Next, take any $s \in \{1, 2, \ldots, n\}$ and let us show that $\hat \phi_i : (\theta^\calD_s, \Boxed{\sqsalidx^\calD}) \to
  (\theta^\calB_s, \Boxed{\sqsalidx^\calB})$ is a rigid surjection.
  
  \medskip
  
  Case $1^\circ$: $s = 1$. The construction of $\sqsubset^\calD$ then ensures that
  $$
    \min \hat \phi_i^{-1}\big((x_1, x_2, \ldots, x_{r_1})\big) = \min (\hat q_i^1)^{-1}\big((x_1, x_2, \ldots, x_{r_1})\big)
  $$
  for every $(x_1, x_2, \ldots, x_{r_1}) \in \theta^\calB_1$, so
  \begin{align*}
    \min \hat \phi_i^{-1}\big((x_1, x_2, \ldots, x_{r_1})\big)
      & = \min (\hat q_i^1)^{-1}\big((x_1, x_2, \ldots, x_{r_1})\big)\\
      & \mathrel{\sqsalidx^\calD} \min (\hat q_i^1)^{-1}\big((y_1, y_2, \ldots, y_{r_1})\big)\\
      & = \min \hat \phi_i^{-1}\big((y_1, y_2, \ldots, y_{r_1})\big)
  \end{align*}
  for all $(x_1, x_2, \ldots, x_{r_1}), (y_1, y_2, \ldots, y_{r_1}) \in \theta^\calB_1$ satisfying
  $(x_1, x_2, \ldots, x_{r_1}) \mathrel{\sqsalidx^\calB} (y_1, y_2, \ldots, y_{r_1})$
  because $q_i^1$ is a strong rigid quotient map.
  
  \medskip

  Case $2^\circ$: $s > 1$. Take $(x_1, x_2, \ldots, x_{r_s}), (y_1, y_2, \ldots, y_{r_s}) \in \theta^\calB_s$
  such that $(x_1, x_2, \ldots, x_{r_s}) \mathrel{\sqsalidx^\calB} (y_1, y_2, \ldots, y_{r_s})$.
  
  \medskip
  
  Case $2.1^\circ$: $(x_1, x_2, \ldots, x_{r_s}) \in \Delta_{B, r_s}$. Then,
  by the construction of~$\sqsubset^\calD$,
  \begin{align*}
    &\min \hat \phi_i^{-1}\big((x_1, x_2, \ldots, x_{r_s})\big) = \min (\hat q_i^1)^{-1}\big((x_1, x_2, \ldots, x_{r_s})\big) \text{ and}\\
    &\min \hat \phi_i^{-1}\big((y_1, y_2, \ldots, y_{r_s})\big) = \min (\hat q_i^t)^{-1}\big((y_1, y_2, \ldots, y_{r_s})\big),
  \end{align*}
  where $t = 1$ if $(y_1, y_2, \ldots, y_{r_s}) \in \Delta_{B, r_s}$, or $t = s$ if
  $(y_1, y_2, \ldots, y_{r_s}) \in \theta^{\calC_s}_s \setminus \Delta_{B, r_s}$.
  If $t = 1$ we are done because $q_i^1$ is a strong rigid quotient map. If, however, $t = s$, we are done
  by the definition of $\sqsalidx^\calD$.

  \medskip

  Case $2.2^\circ$: $(x_1, x_2, \ldots, x_{r_s}) \notin \Delta_{B, r_s}$. Then
  $(y_1, y_2, \ldots, y_{r_s}) \notin \Delta_{B, r_s}$ by definition of~$\sqsalidx^\calD$.
  Therefore,
  \begin{align*}
    &\min \hat \phi_i^{-1}\big((x_1, x_2, \ldots, x_{r_s})\big) = \min (\hat q_i^s)^{-1}\big((x_1, x_2, \ldots, x_{r_s})\big) \text{ and}\\
    &\min \hat \phi_i^{-1}\big((y_1, y_2, \ldots, y_{r_s})\big) = \min (\hat q_i^s)^{-1}\big((y_1, y_2, \ldots, y_{r_s})\big),
  \end{align*}
  and the claim follows because $q_i^s$ is a strong rigid quotient map.

  \medskip
  
  Therefore, $\phi_i : \calD \to \calB$ is a strong rigid quotient map for each~$i$,
  whence follows that $\overline \phi_i : \overline \calD \to \overline \calB$ is a morphism in $\DD$ for each~$i$.
  To complete the proof we still have to show that
  $\overline u \circ \overline \phi_i  = \overline v \circ \overline \phi_j$ whenever
  $\overline u \circ q_i = \overline v \circ q_j$.
  Assume that $\overline u \circ q_i = \overline v \circ q_j$. Take any $x \in D$. Then $x \in C_s$ for some~$s$, so
  $$
    u \circ \phi_i(x) = u \circ q_i^s(x) = v \circ q_j^s(x) = v \circ \phi_j(x),
  $$
  because $\overline u \circ q_i = \overline v \circ q_j$ means that $u \circ q_i^t = v \circ q_j^t$ for each~$t$.
  Therefore, $\overline u \circ \overline \phi_i  = \overline v \circ \overline \phi_j$.
  
  This concludes the proof in case $\Theta$ is a finite relational language.

  \bigskip
  
  Part~II. Assume now that $\Theta$ is an arbitrary relational language satisfying
  $\Boxed\sqsubset \notin \Theta$, and take any $k \ge 2$ and
  $\calA, \calB \in \Ob(\RslSrqm(\Theta, \Boxed{\sqsubset}))$ such that there is a strong rigid
  quotient map $\calB \to \calA$.
  
  Since $\calB$ is a finite $\Theta$-erst, $\theta^\calB = \0$ for every $\theta \in \Theta$ such that $\arity(\theta) > |B|$.
  Moreover, on a finite set there are only finitely many relations whose arities do not exceed $|B|$.
  Therefore, there exists a finite $\Sigma \subseteq \Theta$ such that
  for every $\theta \in \Theta \setminus \Sigma$ we have $\theta^\calB = \0$ or $\theta^\calB = \sigma^\calB$
  for some $\sigma \in \Sigma$. Since there is a strong rigid quotient map $\calB \to \calA$,
  we have the following:
  \begin{itemize}
  \item
    if $\theta^\calB = \0$ for some $\theta \in \Theta \setminus \Sigma$ then $\theta^\calA = \0$, and
  \item
    if $\theta^\calB = \sigma^\calB$ for some $\theta \in \Theta \setminus \Sigma$ and $\sigma \in \Sigma$ then
    $\theta^\calA = \sigma^\calA$.
  \end{itemize}
  
  The category $\RslSrqm(\Sigma, \Boxed{\sqsubset})$ has the dual Ramsey property because $\Sigma$ is finite,
  (Part~I), so there is a $\calC = (C, \Sigma^\calC, \Boxed{\sqsubset^\calC}) \in
  \Ob(\RslSrqm(\Sigma, \Boxed{\sqsubset}))$ such that
  $$
    \calC \longrightarrow (\reduct \calB {\Sigma \union \{\sqsubset\}})^{\reduct \calA {\Sigma \union \{\sqsubset\}}}_k
  $$
  in $\RslSrqm(\Sigma, \Boxed{\sqsubset})^\op$. Define $\calC^* = (C, \Theta^{\calC^*}, \Boxed{\sqsubset^{\calC^*}})
  \in \Ob(\RslSrqm(\Theta, \Boxed{\sqsubset}))$ as follows:
  \begin{itemize}
  \item
    \Boxed{\sqsubset^{\calC^*}} = \Boxed{\sqsubset^{\calC}};
  \item
    if $\sigma \in \Sigma$ let $\sigma^{\calC^*} = \sigma^{\calC}$;
  \item
    if $\theta \in \Theta \setminus \Sigma$ and $\theta^\calB = \0$ let $\theta^{\calC^*} = \0$;
  \item
    if $\theta \in \Theta \setminus \Sigma$ and $\theta^\calB = \sigma^\calB$ for some
    $\sigma \in \Sigma$, let $\theta^{\calC^*} = \sigma^{\calC^*}$.
  \end{itemize}
  Clearly, $\calC^*$ is a $\Theta$-erst and $\calC^* \longrightarrow (\calB)^\calA_k$
  in $\RslSrqm(\Theta, \Boxed{\sqsubset})^\op$ because
  \begin{align*}
    & \hom_{\RslSrqm(\Sigma, \Boxed{\sqsubset})}(\calC, \reduct \calA {\Sigma \union \{\sqsubset\}}) =
      \hom_{\RslSrqm(\Theta, \Boxed{\sqsubset})}(\calC^*, \calA), \text{ and}\\
    & \hom_{\RslSrqm(\Sigma, \Boxed{\sqsubset})}(\calC, \reduct \calB {\Sigma \union \{\sqsubset\}}) =
      \hom_{\RslSrqm(\Theta, \Boxed{\sqsubset})}(\calC^*,\calB).
  \end{align*}
  This concludes the proof.
\end{proof}

\begin{DEF}
  Let $\calA = (A, \Theta^\calA, \Boxed{\sqsubset^\calA})$ and $\calB = (B, \Theta^\calB, \Boxed{\sqsubset^\calB})$ be two
  linearly ordered reflexive $\Theta$-structures. A homomorphism $f : (A, \Theta^\calA) \to (B, \Theta^\calB)$ is a
  \emph{strong rigid quotient map of structures} if the following holds for every $\theta \in \Theta$:
  \begin{itemize}
  \item
    $\hat f : (\theta^\calA, \Boxed{\sqsalidx^\calA}) \to (\theta^\calB, \Boxed{\sqsalidx^\calB})$
    is a rigid surjection; and
  \item
    for every $(x_1, x_2, \ldots, x_r) \in \theta^\calA$,
    if $\restr{f}{\{x_1, x_2, \ldots, x_r\}}$ is not a constant map
    then $x_i \mathrel{\sqsubset^\calA} x_j \Rightarrow
    f(x_i) \mathrel{\sqsubset^\calB} f(x_j)$ for all $i$ and $j$.
  \end{itemize}
\end{DEF}

Let $\RELSrqm(\Theta, \Boxed\sqsubset)$ be the category whose objects are all finite linearly ordered reflexive
$\Theta$-structures and whose morphisms are strong rigid quotient maps of structures. Our final result is a dual
version of the Ne\v set\v ril-R\"odl Theorem.

\begin{THM}[The Dual Ne\v set\v ril-R\"odl Theorem (restricted form)]
  Let $\Theta$ be a relational language and let $\Boxed\sqsubset \notin \Theta$ be a binary relational symbol. Then
  $\RELSrqm(\Theta, \Boxed{\sqsubset})$ has the dual Ramsey property.
\end{THM}
\begin{proof}
  Fix a relational language $\Theta$ such that $\Boxed\sqsubset \notin \Theta$. Let
  $$
    X_\Theta = \big\{ \theta/\sigma : \theta \in \Theta \text{ and } \sigma \text{ is a total quasiorder on } \{1, 2, \ldots, \arity(\theta)\} \big\}
  $$
  be a relational language where $\theta/\sigma$ is a new relational symbol (formally, a pair $(\theta, \sigma)$) such that
  $$
    \arity(\theta/\sigma) = |\{1, 2, \ldots, \arity(\theta)\} / \Boxed{\equiv_\sigma}|.
  $$
  We are going to show that the categories $\RELSrqm(\Theta, \Boxed{\sqsubset})$ and
  $\RslSrqm(X_\Theta, \Boxed{\sqsubset})$ are isomorphic.
  The dual Ramsey property for $\RELSrqm(\Theta, \Boxed{\sqsubset})$ then follows directly from
  Proposition~\ref{dthms.prop.theta-erst}.

  For $\calA = (A, \Theta^\calA, \Boxed{\sqsubset^\calA}) \in \Ob(\RELSrqm(\Theta, \Boxed{\sqsubset}))$
  define a $\calA^\dagger = (A, X_\Theta^{\calA^\dagger}, \Boxed{\sqsubset^{\calA^\dagger}})$ as follows:
  \begin{align*}
    \Boxed{\sqsubset^{\calA^\dagger}} &= \Boxed{\sqsubset^{\calA}},\\
    (\theta/\sigma)^{\calA^\dagger}  &= \Delta_{A, \arity(\theta/\sigma)} \union
        \{\mat(\overline a): \overline a \in \theta^\calA \text{ and } \tp(\overline a) = \sigma \}.
  \end{align*}
  On the other hand, take any
  $\calB = (B, X_\Theta^\calB, \Boxed{\sqsubset^\calB}) \in \Ob(\RslSrqm(X_\Theta, \Boxed\sqsubset))$
  and define $\calB^* = (B, \Theta^{\calB^*}, \Boxed{\sqsubset^{\calB^*}}) \in \Ob(\RELSrqm(\Theta, \Boxed\sqsubset))$ as follows:
  \begin{align*}
    \Boxed{\sqsubset^{\calB^*}} &= \Boxed{\sqsubset^{\calB}},\\
    \theta^{\calB^*} &= \{\tup(\sigma, \overline a) : \sigma \text{ is a total quasiorder on } \{1, 2, \ldots, \arity(\theta)\}\\
                &\qquad\qquad\qquad\qquad\text{and } \overline a \in (\theta/\sigma)^\calB\}.
  \end{align*}
  Consider the functors
  $$
    F : \RELSrqm(\Theta, \Boxed{\sqsubset}) \to \RslSrqm(X_\Theta, \Boxed{\sqsubset}) : \calA \mapsto \calA^\dagger : f \mapsto f
  $$
  and
  $$
    G : \RslSrqm(X_\Theta, \Boxed{\sqsubset}) \to \RELSrqm(\Theta, \Boxed{\sqsubset}) : \calB \mapsto \calB^* : f \mapsto f.
  $$
  Because of \eqref{dthms.eq.tup-mat} we have that $(\calA^\dagger)^* = \calA$
  and $(\calB^*)^\dagger = \calB$ for all $\calA \in \Ob(\RELSrqm(\Theta, \Boxed{\sqsubset}))$ and
  all $\calB \in \Ob(\RslSrqm(X_\Theta, \Boxed{\sqsubset}))$. Hence, $F$ and $G$ are mutually inverse functors,
  so $\RELSrqm(\Theta, \Boxed{\sqsubset})$ and $\RslSrqm(X_\Theta, \Boxed{\sqsubset})$ are isomorphic categories.
  However, we still have to show that $F$ and $G$ are well defined.
  Clearly, both functors are well defined on objects.

  Let us show that $F$ is well defined on morphisms.
  Take any morphism $f : \calA \to \calB$ in $\RELSrqm(\Theta, \Boxed{\sqsubset})$ where
  $\calA = (A, \Theta^\calA, \Boxed{\sqsubset^\calA})$ and $\calB = (B, \Theta^\calB, \Boxed{\sqsubset^\calB})$.

  To see that $f : (A, X_\Theta^{\calA^\dagger}) \to (B, X_\Theta^{\calB^\dagger})$ is a homomorphism, take any
  $\theta \in \Theta$, any total quasiorder $\sigma$ on $\{1, 2, \ldots, \arity(\theta)\}$ and any
  $
    (x_1, x_2, \ldots, x_r) \in (\theta/\sigma)^{\calA^\dagger} \setminus \Delta_{A,r}
  $, where $r = \arity(\theta/\sigma)$.
  Then there exists an $\overline a \in \theta^\calA$ such that $\sigma = \tp(\overline a)$ and
  $(x_1, x_2, \ldots, x_r) = \mat(\overline a)$.
  If $f(x_1) = f(x_2) = \ldots = f(x_r)$ we are done. Assume, therefore, that $\restr{f}{\{x_1, x_2, \ldots, x_r\}}$ is not
  a constant map.
  Because $f$ is a homomorphism of $\Theta$-structures, $\hat f(\overline a) \in \theta^\calB$.
  We also know that $\tp(\hat f(\overline a)) = \tp(\overline a) = \sigma$ (Lemma~\ref{dthm.lem.tp-mat-facts-2}), so
  $\mat(\hat f(\overline a)) \in (\theta/\sigma)^{\calB^\dagger}$. By using Lemma~\ref{dthm.lem.tp-mat-facts-2} again
  we have that $\mat(\hat f(\overline a))
  = \hat f(\mat(\overline a)) = \hat f(x_1, x_2, \ldots, x_r) = (f(x_1), f(x_2), \ldots, f(x_r))$.

  Next, let us show that $\restr{\hat f}{(\theta/\sigma)^{\calA^\dagger}} : ((\theta/\sigma)^{\calA^\dagger}, \Boxed{\sqsalidx^{\calA^\dagger}}) \to
  ((\theta/\sigma)^{\calB^\dagger}, \Boxed{\sqsalidx^{\calB^\dagger}})$ is a rigid surjection for every
  $\theta \in \Theta$ and every total quasiorder $\sigma$ on $\{1, 2, \ldots, \arity(\theta)\}$.
  For notational convenience we let $\hat f_\theta = \restr{\hat f}{\theta^\calA}$ and
  $\hat f_{\theta/\sigma} = \restr{\hat f}{(\theta/\sigma)^{\calA^\dagger}}$.
  Take any $\theta \in \Theta$, any total quasiorder $\sigma$ on $\{1, 2, \ldots, \arity(\theta)\}$ and let
  $r = \arity(\theta/\sigma)$. Note, first, that $\hat f_{\theta/\sigma}$ is surjective because $f$ is a quotient map
  (Lemma~\ref{dthms.lem.strong}) and $\tp(\overline a) = \tp(\hat f(\overline a))$ whenever
  $\overline a = (a_1, a_2, \ldots, a_n) \in \theta^\calA$ and $\restr{f}{\{a_1, a_2, \ldots, a_n\}}$
  is not a constant map (Lemma~\ref{dthm.lem.tp-mat-facts-2}).
  Take any $(x_1, x_2, \ldots, x_r)$, $(y_1, y_2, \ldots, y_r) \in (\theta/\sigma)^{\calB^\dagger}$
  such that $(x_1, x_2, \ldots, x_r) \mathrel{\sqsalidx^{\calB^\dagger}} (y_1, y_2, \ldots, y_r)$.
  
  \medskip
  
  Case $1^\circ$. $|\{x_1, x_2, \ldots, x_r\}| = |\{y_1, y_2, \ldots, y_r\}| = 1$.
  
  \medskip
  
  By Lemma~\ref{dthms.lem.strong}~$(a)$ we have that
  $\min \hat f_{\theta/\sigma}^{-1}(x_1, x_1, \ldots, x_1) = (s, s, \ldots, s)$ where $s = \min f^{-1}(x_1)$ and
  $\min \hat f_{\theta/\sigma}^{-1}(y_1, y_1, \ldots, y_1) = (t, t, \ldots, t)$ where $t = \min f^{-1}(y_1)$.
  Since $(x_1, x_1, \ldots, x_1) \mathrel{\sqsalidx^{\calB^\dagger}} (y_1, y_1, \ldots, y_1)$,
  we know that $x_1 \mathrel{\sqsubset^\calB} y_1$, so $s \mathrel{\sqsubset^\calA} t$ because
  $f$ is a rigid surjection $(A, \Boxed{\sqsubset^\calA}) \to (B, \Boxed{\sqsubset^\calB})$.
  But then
  \begin{align*}
    \min \hat f_{\theta/\sigma}^{-1}(x_1, x_1, \ldots, x_1)
        &= (s, s, \ldots, s)\\
        &\mathrel{\sqsalidx^{\calA^\dagger}} (t, t, \ldots, t) = \min \hat f_{\theta/\sigma}^{-1}(y_1, y_1, \ldots, y_1).
  \end{align*}

  \medskip

  Case $2^\circ$. $|\{x_1, x_2, \ldots, x_r\}| = 1$ and $|\{y_1, y_2, \ldots, y_r\}| > 1$.
  
  \medskip
  
  Then $\min \hat f_{\theta/\sigma}^{-1}(x_1, x_1, \ldots, x_1) = (s, s, \ldots, s)$ and
  $\min \hat f_{\theta/\sigma}^{-1}(y_1, y_2, \ldots, y_r) = (t_1, t_2, \ldots, t_r)$
  where $|\{t_1, t_2, \ldots, t_r\}| > 1$, so
  $$
    \tp(\min \hat f_{\theta/\sigma}^{-1}(x_1, x_1, \ldots, x_1))
    \triangleleft \tp(\min \hat f_{\theta/\sigma}^{-1}(y_1, y_2, \ldots, y_r)),
  $$
  whence
  $\min \hat f_{\theta/\sigma}^{-1}(x_1, x_1, \ldots, x_1)
  \mathrel{\sqsalidx^{\calA^\dagger}} \min \hat f_{\theta/\sigma}^{-1}(y_1, y_2, \ldots, y_r)$.

  \medskip
  
  Case $3^\circ$. $|\{x_1, x_2, \ldots, x_r\}| > 1$ and $|\{y_1, y_2, \ldots, y_r\}| > 1$.
  
  \medskip

  Let $(x_1, x_2, \ldots, x_r) = \mat(\overline p)$ and $(y_1, y_2, \ldots, y_r) = \mat(\overline q)$
  for some $\overline p, \overline q \in \theta^\calB$ such that $\tp(\overline p) = \tp(\overline q) = \sigma$.
  Since
  $$
    \mat(\overline p) = (x_1, x_2, \ldots, x_r) \mathrel{\sqsalidx^{\calB^\dagger}} (y_1, y_2, \ldots, y_r) =
    \mat(\overline q)
  $$
  and $\tp(\overline p) = \tp(\overline q)$ we have that
  $\overline p \mathrel{\sqsalidx^\calB} \overline q$ (Lemma~\ref{dthm.lem.tp-mat-facts}).
  Since $\hat f_\theta : (\theta^\calA, \Boxed{\sqsalidx^\calA}) \to (\theta^\calB, \Boxed{\sqsalidx^\calB})$ is a rigid
  surjection, we know that $\min \hat f_\theta^{-1}(\overline p) \mathrel{\sqsalidx^\calA} \min \hat f_\theta^{-1}(\overline q)$.
  On the other hand, $\tp(\min \hat f_\theta^{-1}(\overline p)) = \tp(\overline a) = \tp(\overline p) = \tp(\overline q)
  = \tp(\overline b) = \tp(\min \hat f_\theta^{-1}(\overline q))$ for some $\overline a \in \hat f_\theta^{-1}(\overline p)$ and
  $\overline b \in \hat f_\theta^{-1}(\overline q)$ where the minimum is achieved, so by Lemma~\ref{dthm.lem.tp-mat-facts}
  we conclude that
  $$
    \mat(\min \hat f_\theta^{-1}(\overline p)) \mathrel{\sqsalidx^{\calA^\dagger}}
    \mat(\min \hat f_\theta^{-1}(\overline q)).
  $$
  By Lemma~\ref{dthms.lem.AUX} we finally get
  $\min \hat f_{\theta/\sigma}^{-1}(\mat(\overline p))
  \mathrel{\sqsalidx^{\calA^\dagger}} \min \hat f_{\theta/\sigma}^{-1}(\mat(\overline q))$, that is,
  $
    \min \hat f_{\theta/\sigma}^{-1}(x_1, x_2, \ldots, x_r) \mathrel{\sqsalidx^{\calA^\dagger}}
    \min \hat f_{\theta/\sigma}^{-1}(y_1, y_2, \ldots, y_r)
  $.
  This concludes the proof of Case~$3^\circ$ and the proof that the functor $F$ is well defined on morphisms.

  \medskip

  Let us show that $G$ is well defined on morphisms.
  Take any morphism $f : \calA \to \calB$ in $\RslSrqm(X_\Theta, \Boxed{\sqsubset})$ where
  $\calA = (A, X_\Theta^\calA, \Boxed{\sqsubset^\calA})$ and $\calB = (B, X_\Theta^\calB, \Boxed{\sqsubset^\calB})$.

  Let us first show that $f : (A, \Theta^{\calA^*}) \to (B, \Theta^{\calB^*})$ is a homomorphism.
  Take any $\theta \in \Theta$ and any $\overline x = (x_1, x_2, \ldots, x_n) \in \theta^{\calA^*}$.
  If $\restr{f}{\{x_1, x_2, \ldots, x_n\}}$ is a constant map we are done. Assume, therefore, that this is not the case.
  Then $\overline x = \tup(\sigma, \overline a)$ for some $\sigma$ and some $\overline a \in (\theta / \sigma)^\calA$.
  Because $f : \calA \to \calB$ is a homomorphism, $\hat f(\overline a) \in (\theta/\sigma)^\calB$, whence
  $\tup(\sigma, \hat f(\overline a)) \in \theta^{\calB^*}$.
  Therefore, $\hat f(\overline x) = \hat f(\tup(\sigma, \overline a)) = \tup(\sigma, \hat f(\overline a)) \in \theta^{\calB^*}$,
  using Lemma~\ref{dthm.lem.tp-mat-facts-2} for the second equality.
  
  \medskip
  
  Next, let us show that for every $\theta \in \Theta$ and every $(x_1, x_2, \ldots, x_n) \in \theta^{\calA^*}$,
  if $\restr{f}{\{x_1, x_2, \ldots, x_n\}}$ is not a constant map
  then $x_i \mathrel{\sqsubset^{\calA^*}} x_j \Rightarrow
  f(x_i) \mathrel{\sqsubset^{\calB^*}} f(x_j)$ for all $i$ and $j$.
  Take any $\theta \in \Theta$, any $\overline x = (x_1, x_2, \ldots, x_n) \in \theta^{\calA^*}$ and assume that
  $\restr{f}{\{x_1, x_2, \ldots, x_n\}}$ is not a constant map. By definition of $\theta^{\calA^*}$ we have that
  $\overline x = \tup(\sigma, \overline a)$ for some $\sigma$ and some $\overline a \in (\theta / \sigma)^\calA$.
  Clearly, $\mat(\overline x) = \overline a$. Assume that $x_i \mathrel{\sqsubset^{\calA^*}} x_j$.
  Then $x_i \mathrel{\sqsubset^{\calA}} x_j$, so
  $$
    \overline a = \mat(\overline x) = (\ldots, x_i, \ldots, x_j, \ldots) \in (\theta / \sigma)^\calA.
  $$
  Because $f : \calA \to \calB$ is a homomorphism,
  $$
    (\ldots, f(x_i), \ldots, f(x_j), \ldots) \in (\theta / \sigma)^\calB,
  $$
  so $f(x_i) \mathrel{\sqsubset^{\calB}} f(x_j)$, or, equivalently, $f(x_i) \mathrel{\sqsubset^{\calB^*}} f(x_j)$.

  Finally, let us show that $\restr{\hat f}{\theta^{\calA^*}} : (\theta^{\calA^*}, \Boxed{\sqsalidx^{\calA^*}}) \to
  (\theta^{\calB^*}, \Boxed{\sqsalidx^{\calB^*}})$ is a rigid surjection for every $\theta \in \Theta$.
  For notational convenience, this time we let $\hat f_\theta = \restr{\hat f}{\theta^{\calA^*}}$ and
  $\hat f_{\theta/\sigma} = \restr{\hat f}{(\theta/\sigma)^{\calA}}$.
  
  Note, first, that $\hat f_\theta : \theta^{\calA^*} \to \theta^{\calB^*}$ is surjective because so is
  $\hat f_{\theta/\sigma} : (\theta/\sigma)^{\calA} \to (\theta/\sigma)^{\calB}$ for every~$\sigma$.
  Take any $\overline x, \overline y \in \theta^{\calB^*}$ such that
  $\overline x \mathrel{\sqsalidx^{\calB^*}} \overline y$. Let
  $\overline x = (x_1, x_2, \ldots, x_n)$ and $\overline y = (y_1, y_2, \ldots, y_n)$.

  \medskip
  
  Case $1^\circ$. $|\{x_1, x_2, \ldots, x_n\}| = |\{y_1, y_2, \ldots, y_n\}| = 1$.
  
  \medskip
  
  By Lemma~\ref{dthms.lem.strong}~$(a)$ we have that
  $\min \hat f_\theta^{-1}(x_1, x_1, \ldots, x_1) = (s, s, \ldots, s)$ where $s = \min f^{-1}(x_1)$ and
  $\min \hat f_\theta^{-1}(y_1, y_1, \ldots, y_1) = (t, t, \ldots, t)$ where $t = \min f^{-1}(y_1)$.
  Since $(x_1, x_1, \ldots, x_1) \mathrel{\sqsalidx^{\calB^*}} (y_1, y_1, \ldots, y_1)$,
  we know that $x_1 \mathrel{\sqsubset^\calB} y_1$, so $s \mathrel{\sqsubset^\calA} t$ because
  $f$ is a rigid surjection $(A, \Boxed{\sqsubset^\calA}) \to (B, \Boxed{\sqsubset^\calB})$.
  But then
  \begin{align*}
    \min \hat f_\theta^{-1}(x_1, x_1, \ldots, x_1)
      &= (s, s, \ldots, s)\\
      &\mathrel{\sqsalidx^{\calA^*}} (t, t, \ldots, t) = \min \hat f_\theta^{-1}(y_1, y_1, \ldots, y_1).
  \end{align*}

  \medskip

  Case $2^\circ$. $|\{x_1, x_2, \ldots, x_n\}| = 1$ and $|\{y_1, y_2, \ldots, y_n\}| > 1$.
  
  \medskip
  
  Then $\min \hat f_\theta^{-1}(x_1, x_1, \ldots, x_1) = (s, s, \ldots, s)$ and
  $\min \hat f_\theta^{-1}(y_1, y_2, \ldots, y_n) = (t_1, t_2, \ldots, t_n)$ where $|\{t_1, t_2, \ldots, t_n\}| > 1$, so
  $$
    \tp(\min \hat f_\theta^{-1}(x_1, x_1, \ldots, x_1)) \triangleleft \tp(\min \hat f_\theta^{-1}(y_1, y_2, \ldots, y_n)),
  $$
  whence
  $\min \hat f_\theta^{-1}(x_1, x_1, \ldots, x_1) \mathrel{\sqsalidx^{\calA^*}} \min \hat f_\theta^{-1}(y_1, y_2, \ldots, y_n)$.

  \medskip
  
  Case $3^\circ$. $|\{x_1, x_2, \ldots, x_n\}| > 1$ and $|\{y_1, y_2, \ldots, y_n\}| > 1$.
  
  \medskip
  
  By definition of $\theta^{\calB^*}$ we have that
  $\overline x = \tup(\sigma, \overline a)$ for some $\sigma$ and some $\overline a \in (\theta / \sigma)^\calB$ and
  $\overline y = \tup(\tau, \overline b)$ for some $\tau$ and some $\overline b \in (\theta / \tau)^\calB$.
  
  Assume, first, that $\sigma \ne \tau$. Then $\overline x \mathrel{\sqsalidx^{\calB^*}} \overline y$ actually means that
  $\tp(\overline x) \triangleleft \tp(\overline y)$.
  Let $\min \hat f_\theta^{-1}(\overline x) = \overline u \in \theta^{\calA^*}$
  and $\min \hat f_\theta^{-1}(\overline y) = \overline v \in \theta^{\calA^*}$.
  Lemma~\ref{dthm.lem.tp-mat-facts-2} then yields that
  $\tp(\overline u) = \tp(\overline x) \triangleleft \tp(\overline y) = \tp(\overline v)$, so
  $\overline u \mathrel{\sqsalidx^{\calA^*}} \overline v$.
  
  Assume, now, that $\sigma = \tau$. Then $\tp(\overline x) = \tp(\overline y)$, so Lemma~\ref{dthm.lem.tp-mat-facts}
  implies that $\overline a = \mat(\overline x) \mathrel{\sqsalidx^\calB} \mat(\overline y) = \overline b$.
  Since $\hat f_{\theta/\sigma} : ((\theta / \sigma)^{\calA}, \Boxed{\sqsalidx^\calA}) \to ((\theta / \sigma)^{\calB}, \Boxed{\sqsalidx^\calB})$
  is a rigid surjection, it follows that $\min \hat f_{\theta/\sigma}^{-1}(\mat(\overline x)) \mathrel{\sqsalidx^\calA}
  \min \hat f_{\theta/\sigma}^{-1}(\mat(\overline y))$. Lemma~\ref{dthms.lem.AUX} yields that
  \begin{align*}
    \min \hat f_{\theta/\sigma}^{-1}(\mat(\overline x)) &= \mat(\min \hat f_\theta^{-1}(\overline x)) \text{ and}\\
    \min \hat f_{\theta/\sigma}^{-1}(\mat(\overline y)) &= \mat(\min \hat f_\theta^{-1}(\overline y)).
  \end{align*}
  Therefore,
  $\mat(\min \hat f_\theta^{-1}(\overline x)) \mathrel{\sqsalidx^\calA} \mat(\min \hat f_\theta^{-1}(\overline y))$.
  On the other hand, $\tp(\overline x) = \tp(\overline y)$ implies that
  $\tp(\min \hat f_\theta^{-1}(\overline x)) = \tp(\min \hat f_\theta^{-1}(\overline y))$.
  Lemma~\ref{dthm.lem.tp-mat-facts} then ensures that
  $\min \hat f_\theta^{-1}(\overline x) \mathrel{\sqsalidx^{\calA^*}} \min \hat f_\theta^{-1}(\overline y)$.
  This concludes the proof of Case~$3^\circ$, the proof that the functor $G$ is well defined on morphisms, and
  the proof of the theorem.
\end{proof}

\section{Tournaments -- a non-example}
\label{dthms.sec.no-drp-tournaments}

In this section we prove that the category whose objects are finite linearly ordered reflexive tournaments
and whose morphisms are rigid surjective homomorphisms does not have the dual Ramsey property.

A \emph{linearly ordered reflexive tournament} is a structure $(A, \Boxed\to, \Boxed\sqsubset)$ where
$\sqsubset$ is a linear order on $A$ and $\Boxed\to \subseteq A^2$ is a reflexive relation such that for all $x \ne y$,
either $x \to y$ or $y \to x$. A mapping $f : A \to A'$ is a \emph{rigid surjective homomorphism}
from $\calA = (A, \Boxed\to, \Boxed\sqsubset)$ to $\calA' = (A', \Boxed{\to'}, \Boxed{\sqsubset'})$
if $f : (A, \Boxed\to) \to (A', \Boxed{\to'})$ is a homomorphism and
$f : (A, \Boxed\sqsubset) \to (A', \Boxed{\sqsubset'})$ is a rigid surjection.

A reflexive tournament $(T, \Boxed\to)$ is an \emph{inflation} of a reflexive tournament $(S, \Boxed\to)$
if there exists a surjective homomorphism $(T, \Boxed\to) \to (S, \Boxed\to)$.
Finite reflexive tournaments $(S_1, \Boxed\to)$ and $(S_2, \Boxed\to)$ are \emph{siblings} if
there exists a finite reflexive tournament $(T, \Boxed\to)$ which is an inflation of $(S_1, \Boxed\to)$ and
an inflation of $(S_2, \Boxed\to)$. Let $C_3$ denote the reflexive tournament $(\{1, 2, 3\}, \Boxed\to)$ whose
nontrivial edges are $1 \to 2$, $2 \to 3$ and $3 \to 1$, and let $C_3^+$
denote the reflexive tournament $(\{1, 2, 3, 4\}, \Boxed\to)$ whose nontrivial edges are $1 \to 2$, $2 \to 3$, $3 \to 1$,
$1 \to 4$, $2 \to 4$ and $3 \to 4$.

\begin{LEM}\label{dthms.lem.tmts}
  $C_3$ and $C_3^+$ are not siblings.
\end{LEM}
\begin{proof}
  Suppose, to the contrary, that there is a finite reflexive tournament $T$ and surjective homomorphisms
  $f : T \to C_3$ and $g : T \to C_3^+$.
  Let $A_i = f^{-1}(i)$, $1 \le i \le 3$, and $B_j = f^{-1}(j)$, $1 \le j \le 4$.
  Let $D_{ij} = A_i \sec B_j$, $1 \le i \le 3$, $1 \le j \le 4$.
  For each $j$, it is not possible that all of the sets $D_{1j}$, $D_{2j}$ and $D_{3j}$ are empty because
  $B_j = D_{1j} \union D_{2j} \union D_{3j}$ is nonempty. Analogously,
  for each $i$, it is not possible that all of the sets $D_{i1}$, $D_{i2}$, $D_{i3}$ and $D_{i4}$ are empty.
  
  Now, consider $D_{ij}$ and $D_{uv}$ for some
  $1 \le i, u \le 3$ and $1 \le j, v \le 4$, and note that if $i \to u$ in $C_3$ and $v \to j$ in
  $C_3^+$ then $D_{ij} = \0$ or $D_{uv} = \0$. (If this is not the case, take arbitrary $x \in D_{ij}$ and $y \in D_{uv}$.
  If $x \to y$ in $T$ then $g(x) \to g(y)$ in $C_3^+$. But $g(x) = j$ because $x \in D_{ij} \subseteq B_j$ and
  $g(y) = v$ because $y \in D_{uv} \subseteq B_v$. Hence, $j \to v$, which contradicts the assumption. The other possibility,
  $y \to x$ in $T$, leads analogously to the contradiction with $i \to u$ in $C_3$.)
  We shall say that $(ij, uv)$ is a critical pair if $i \to u$ in $C_3$ and $v \to j$ in
  $C_3^+$. It is easy to list all the critical pairs:
  \begin{gather*}
    (11,23),\;
    (11,32),\;
    (11,34),\;
    (12,21),\;
    (12,33),\;
    (12,34),\\
    (13,22),\,
    (13,31),\;
    (13,34),\;
    (14,21),\;
    (14,22),\;
    (14,23),\\
    (21,33),\;
    (22,31),\,
    (23,32),\;
    (24,31),\;
    (24,32),\;
    (24,33).
  \end{gather*}
  Let $M = [m_{ij}]_{3 \times 4}$ be a 01-matrix such that:
  $$
    m_{ij} = \begin{cases}
      0, & D_{ij} = \0,\\
      1, & D_{ij} \ne \0,
    \end{cases}
  $$
  where $1 \le i \le 3$ and $1 \le j \le 4$.
  Then, as we have just seen, $M$ has the following properties:
  \begin{enumerate}
  \item
    each row contains at least one occurrence of~1;
  \item
    each column contains at least one occurrence of~1; and
  \item
    for every critical pair $(ij, uv)$ we have that $m_{ij} = 0$ or $m_{uv} = 0$ (or both).
  \end{enumerate}
  Let us show that no 01-matrix $M = [m_{ij}]_{3 \times 4}$ satisfies all the three properties.
  There are only seven possibilities to fill the first column by 0's and 1's (the option 000 is excluded by~(ii)).
  Let us consider only the case 100, Fig.~\ref{dthms.fig.siblings}~$(a)$,
  as the other cases follow by analogous arguments. The entries 23, 32 and 34 have to be~0 because of~(iii)
  and the critical pairs $(11,23)$, $(11,32)$ and $(11,34)$, Fig.~\ref{dthms.fig.siblings}~$(b)$.
  Then the entry 33 has to be~1 because of (i), so~(iii) and the critical pairs $(12,33)$ and $(24,33)$ force
  the entries 12 and 24 to be~0, Fig.~\ref{dthms.fig.siblings}~$(c)$. Using (i) once more, the entry 22 has to be~1,
  and the critical pair $(14,22)$ forces the entry 14 to be~0, Fig.~\ref{dthms.fig.siblings}~$(d)$.
  Now, the last column of the matrix is 000, which contradicts~(ii).

\begin{figure}
  $$
    \begin{array}{cccccccccc}
      \begin{array}{|c|c|c|c|}
        \hline
          1 & \  & \  & \  \\ \hline
          0 & \  & \  & \  \\ \hline
          0 & \  & \  & \  \\ \hline
      \end{array}
    &
      \begin{array}{|c|c|c|c|}
        \hline
          1 & \  & \  & \  \\ \hline
          0 & \  & 0  & \  \\ \hline
          0 & 0  & \  & 0  \\ \hline
      \end{array}
    &
      \begin{array}{|c|c|c|c|}
        \hline
          1 & 0  & \  & \  \\ \hline
          0 & \  & 0  & 0  \\ \hline
          0 & 0  & 1  & 0  \\ \hline
      \end{array}
    &
      \begin{array}{|c|c|c|c|}
        \hline
          1 & 0  & \  & 0  \\ \hline
          0 & 1  & 0  & 0  \\ \hline
          0 & 0  & 1  & 0  \\ \hline
      \end{array}
    \\
      (a) & (b) & (c) & (d)
    \end{array}
  $$
  \caption{$C_3$ and $C_3^+$ are not siblings}
  \label{dthms.fig.siblings}
\end{figure}
  
  Therefore, the assumption that there is a finite reflexive tournament $T$ and surjective homomorphisms
  $f : T \to C_3$ and $g : T \to C_3^+$ leads to a contradiction.
\end{proof}

\begin{THM}\label{dthms.thm.no-drp-tournaments}
  Let $\TT$ be a category whose objects are finite linearly ordered reflexive tournaments and whose morphisms are
  rigid surjective homomorphisms. Then $\TT$ does not have the dual Ramsey property.
\end{THM}
\begin{proof}
  Let $\calA = (A, \Boxed\to, \Boxed<)$ and $\calB = (B, \Boxed\to, \Boxed<)$
  be linearly ordered tournaments depicted in Fig.~\ref{dthms.fig.tmts}, where
  $A = \{1, 2\}$, $B = \{1, 2, 3, 4, 5, 6, 7\}$ and~$<$ is the ordering of the integers.
  Let us show that no finite linearly ordered reflexive tournament $\calT$ satisfies $\calT \longrightarrow (\calB)^\calA_2$
  in $\TT^\op$. Take any finite linearly ordered reflexive tournament 
  $\calT = (T, \Boxed\to, \Boxed\sqsubset)$ and define the coloring
  $\chi : \hom_\TT(\calT, \calA) \to \{1, 2\}$ as follows:
  $$
    \chi(f) = \begin{cases}
      1, & \begin{array}[t]{@{}l}
             \text{the subtournament of } (T, \Boxed\to) \text{ induced by } f^{-1}(1) \text{ is}\\
             \text{an inflation of } C_3,
           \end{array}\\
      2, & \text{otherwise}.
    \end{cases}
  $$
  Let $\phi, \psi : \{1, 2, 3, 4, 5, 6, 7\} \to \{1, 2\}$ be the following maps:
  $$
    \phi = \begin{pmatrix}
      1 & 2 & 3 & 4 & 5 & 6 & 7 \\
      1 & 1 & 1 & 2 & 2 & 2 & 2
    \end{pmatrix}
    \text{ and }
    \psi = \begin{pmatrix}
      1 & 2 & 3 & 4 & 5 & 6 & 7 \\
      1 & 1 & 1 & 1 & 2 & 2 & 2
    \end{pmatrix}.
  $$
  Clearly, $\phi, \psi \in \hom_\TT(\calB, \calA)$. Now, take any $w \in \hom_\TT(\calT, \calB)$.
  Since
  $$
    (\phi \circ w)^{-1}(1) = w^{-1}(\phi^{-1}(1)) = w^{-1}(\{1, 2, 3\}),
  $$
  it follows that
  $(\phi \circ w)^{-1}$ induces an inflation of $C_3$ in $(T, \Boxed\to)$, so $\chi(\phi \circ w) = 1$.
  Let us show that $\chi(\psi \circ w) = 2$. Suppose this is not the case. Then
  $\chi(\psi \circ w) = 1$, whence follows that
  $
    (\psi \circ w)^{-1} = w^{-1}(\{1, 2, 3, 4\})
  $
  induces a subtournament $(S, \Boxed\to)$ of $(T, \Boxed\to)$ which is an inflation of $C_3$.
  But the subtournament of $(B, \Boxed\to)$ induced by $\{1, 2, 3, 4\}$ is $C_3^+$, whence follows that
  $(S, \Boxed\to)$ is at the same time an inflation of $C_3$ -- contradiction with Lemma~\ref{dthms.lem.tmts}.
\end{proof}

\begin{figure}
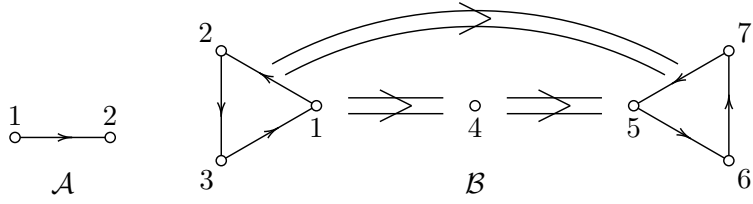

  \centering
  \input tmts.pgf
  \caption{The tournaments in the proof of Theorem~\ref{dthms.thm.no-drp-tournaments}}
  \label{dthms.fig.tmts}
\end{figure}

\section{Acknowledgements}

The author gratefully acknowledges the support of the Grant No.\ 174019 of the Ministry of Education,
Science and Technological Development of the Republic of Serbia.

\end{document}

%% file: srqm3.pgf
%% DO NOT FORGET TO \usepackage{pgf,acadpgf}
\begin{pgfpicture}
  \pgfsetxvec{\pgfpoint{\acadpgfunit}{0pt}}
  \pgfsetyvec{\pgfpoint{0pt}{\acadpgfunit}}
  \pgfsetlinewidth{\acadpgflinewidth}
  \pgftransformshift{\pgfpointxy{-45.1318}{-211.123}}

  \begin{pgfscope}
    \pgfpathmoveto{\pgfpointxy{550.0}{700.0}}
    \pgfpathlineto{\pgfpointxy{650.0}{700.0}}
    \pgfusepath{stroke}
  \end{pgfscope}
  \begin{pgfscope}
    \pgfpathmoveto{\pgfpointxy{750.0}{700.0}}
    \pgfpathlineto{\pgfpointxy{850.0}{700.0}}
    \pgfusepath{stroke}
  \end{pgfscope}
  \begin{pgfscope}
    \pgfpathmoveto{\pgfpointxy{350.0}{400.0}}
    \pgfpathlineto{\pgfpointxy{450.0}{400.0}}
    \pgfusepath{stroke}
  \end{pgfscope}
  \begin{pgfscope}
    \pgfpathmoveto{\pgfpointxy{550.0}{400.0}}
    \pgfpathlineto{\pgfpointxy{650.0}{400.0}}
    \pgfusepath{stroke}
  \end{pgfscope}
  \begin{pgfscope}
    \pgfsetdash{{1.5pt}{2pt}}{0pt}
    \pgfpathmoveto{\pgfpointxy{350.0}{700.0}}
    \pgfpathlineto{\pgfpointxy{350.0}{400.0}}
    \pgfusepath{stroke}
  \end{pgfscope}
  \begin{pgfscope}
    \pgfpathmoveto{\pgfpointxy{356.029}{454.622}}
    \pgfpatharcaxes{150.0}{180.0}{\pgfpointxy{45.0}{0.0}}{\pgfpointxy{0.0}{45.0}}
    \pgfusepath{stroke}
  \end{pgfscope}
  \begin{pgfscope}
    \pgfpathmoveto{\pgfpointxy{350.0}{432.122}}
    \pgfpatharcaxes{0.0}{30.0}{\pgfpointxy{45.0}{0.0}}{\pgfpointxy{0.0}{45.0}}
    \pgfusepath{stroke}
  \end{pgfscope}
  \begin{pgfscope}
    \pgfsetdash{{1.5pt}{2pt}}{0pt}
    \pgfpathmoveto{\pgfpointxy{450.0}{700.0}}
    \pgfpathlineto{\pgfpointxy{450.0}{400.0}}
    \pgfusepath{stroke}
  \end{pgfscope}
  \begin{pgfscope}
    \pgfpathmoveto{\pgfpointxy{456.029}{454.622}}
    \pgfpatharcaxes{150.0}{180.0}{\pgfpointxy{45.0}{0.0}}{\pgfpointxy{0.0}{45.0}}
    \pgfusepath{stroke}
  \end{pgfscope}
  \begin{pgfscope}
    \pgfpathmoveto{\pgfpointxy{450.0}{432.122}}
    \pgfpatharcaxes{0.0}{30.0}{\pgfpointxy{45.0}{0.0}}{\pgfpointxy{0.0}{45.0}}
    \pgfusepath{stroke}
  \end{pgfscope}
  \begin{pgfscope}
    \pgfsetdash{{1.5pt}{2pt}}{0pt}
    \pgfpathmoveto{\pgfpointxy{550.0}{700.0}}
    \pgfpathlineto{\pgfpointxy{550.0}{400.0}}
    \pgfusepath{stroke}
  \end{pgfscope}
  \begin{pgfscope}
    \pgfpathmoveto{\pgfpointxy{556.029}{450.319}}
    \pgfpatharcaxes{150.0}{180.0}{\pgfpointxy{45.0}{0.0}}{\pgfpointxy{0.0}{45.0}}
    \pgfusepath{stroke}
  \end{pgfscope}
  \begin{pgfscope}
    \pgfpathmoveto{\pgfpointxy{550.0}{427.819}}
    \pgfpatharcaxes{0.0}{30.0}{\pgfpointxy{45.0}{0.0}}{\pgfpointxy{0.0}{45.0}}
    \pgfusepath{stroke}
  \end{pgfscope}
  \begin{pgfscope}
    \pgfsetdash{{1.5pt}{2pt}}{0pt}
    \pgfpathmoveto{\pgfpointxy{650.0}{700.0}}
    \pgfpathlineto{\pgfpointxy{650.0}{400.0}}
    \pgfusepath{stroke}
  \end{pgfscope}
  \begin{pgfscope}
    \pgfpathmoveto{\pgfpointxy{656.029}{450.319}}
    \pgfpatharcaxes{150.0}{180.0}{\pgfpointxy{45.0}{0.0}}{\pgfpointxy{0.0}{45.0}}
    \pgfusepath{stroke}
  \end{pgfscope}
  \begin{pgfscope}
    \pgfpathmoveto{\pgfpointxy{650.0}{427.819}}
    \pgfpatharcaxes{0.0}{30.0}{\pgfpointxy{45.0}{0.0}}{\pgfpointxy{0.0}{45.0}}
    \pgfusepath{stroke}
  \end{pgfscope}
  \begin{pgfscope}
    \pgfsetdash{{1.5pt}{2pt}}{0pt}
    \pgfpathmoveto{\pgfpointxy{750.0}{700.0}}
    \pgfpathlineto{\pgfpointxy{350.0}{400.0}}
    \pgfusepath{stroke}
  \end{pgfscope}
  \begin{pgfscope}
    \pgfpathmoveto{\pgfpointxy{618.098}{593.537}}
    \pgfpatharcaxes{96.8699}{126.87}{\pgfpointxy{45.0}{0.0}}{\pgfpointxy{0.0}{45.0}}
    \pgfusepath{stroke}
  \end{pgfscope}
  \begin{pgfscope}
    \pgfpathmoveto{\pgfpointxy{596.48}{584.86}}
    \pgfpatharcaxes{306.87}{336.87}{\pgfpointxy{45.0}{0.0}}{\pgfpointxy{0.0}{45.0}}
    \pgfusepath{stroke}
  \end{pgfscope}
  \begin{pgfscope}
    \pgfsetdash{{1.5pt}{2pt}}{0pt}
    \pgfpathmoveto{\pgfpointxy{850.0}{700.0}}
    \pgfpathlineto{\pgfpointxy{450.0}{400.0}}
    \pgfusepath{stroke}
  \end{pgfscope}
  \begin{pgfscope}
    \pgfpathmoveto{\pgfpointxy{718.098}{593.537}}
    \pgfpatharcaxes{96.8699}{126.87}{\pgfpointxy{45.0}{0.0}}{\pgfpointxy{0.0}{45.0}}
    \pgfusepath{stroke}
  \end{pgfscope}
  \begin{pgfscope}
    \pgfpathmoveto{\pgfpointxy{696.48}{584.86}}
    \pgfpatharcaxes{306.87}{336.87}{\pgfpointxy{45.0}{0.0}}{\pgfpointxy{0.0}{45.0}}
    \pgfusepath{stroke}
  \end{pgfscope}
  \begin{pgfscope}
    \pgfsetdash{{1.5pt}{2pt}}{0pt}
    \pgfpathmoveto{\pgfpointxy{225.0}{675.0}}
    \pgfpathlineto{\pgfpointxy{225.0}{425.0}}
    \pgfusepath{stroke}
  \end{pgfscope}
  \begin{pgfscope}
    \pgfpathmoveto{\pgfpointxy{231.029}{447.5}}
    \pgfpatharcaxes{150.0}{180.0}{\pgfpointxy{45.0}{0.0}}{\pgfpointxy{0.0}{45.0}}
    \pgfusepath{stroke}
  \end{pgfscope}
  \begin{pgfscope}
    \pgfpathmoveto{\pgfpointxy{225.0}{425.0}}
    \pgfpatharcaxes{0.0}{30.0}{\pgfpointxy{45.0}{0.0}}{\pgfpointxy{0.0}{45.0}}
    \pgfusepath{stroke}
  \end{pgfscope}
  \begin{pgfscope}
    \pgfpathmoveto{\pgfpointxy{350.0}{700.0}}
    \pgfpathlineto{\pgfpointxy{325.0}{725.0}}
    \pgfusepath{stroke}
  \end{pgfscope}
  \begin{pgfscope}
    \pgfpathmoveto{\pgfpointxy{350.0}{700.0}}
    \pgfpathlineto{\pgfpointxy{375.0}{725.0}}
    \pgfusepath{stroke}
  \end{pgfscope}
  \begin{pgfscope}
    \pgfpathmoveto{\pgfpointxy{375.0}{725.0}}
    \pgfpatharcaxes{-45.0}{225.0}{\pgfpointxy{35.3553}{0.0}}{\pgfpointxy{0.0}{35.3553}}
    \pgfusepath{stroke}
  \end{pgfscope}
  \begin{pgfscope}
    \pgfpathmoveto{\pgfpointxy{450.0}{700.0}}
    \pgfpathlineto{\pgfpointxy{425.0}{725.0}}
    \pgfusepath{stroke}
  \end{pgfscope}
  \begin{pgfscope}
    \pgfpathmoveto{\pgfpointxy{450.0}{700.0}}
    \pgfpathlineto{\pgfpointxy{475.0}{725.0}}
    \pgfusepath{stroke}
  \end{pgfscope}
  \begin{pgfscope}
    \pgfpathmoveto{\pgfpointxy{475.0}{725.0}}
    \pgfpatharcaxes{-45.0}{225.0}{\pgfpointxy{35.3553}{0.0}}{\pgfpointxy{0.0}{35.3553}}
    \pgfusepath{stroke}
  \end{pgfscope}
  \begin{pgfscope}
    \pgfpathmoveto{\pgfpointxy{550.0}{700.0}}
    \pgfpathlineto{\pgfpointxy{525.0}{725.0}}
    \pgfusepath{stroke}
  \end{pgfscope}
  \begin{pgfscope}
    \pgfpathmoveto{\pgfpointxy{550.0}{700.0}}
    \pgfpathlineto{\pgfpointxy{575.0}{725.0}}
    \pgfusepath{stroke}
  \end{pgfscope}
  \begin{pgfscope}
    \pgfpathmoveto{\pgfpointxy{575.0}{725.0}}
    \pgfpatharcaxes{-45.0}{225.0}{\pgfpointxy{35.3553}{0.0}}{\pgfpointxy{0.0}{35.3553}}
    \pgfusepath{stroke}
  \end{pgfscope}
  \begin{pgfscope}
    \pgfpathmoveto{\pgfpointxy{650.0}{700.0}}
    \pgfpathlineto{\pgfpointxy{625.0}{725.0}}
    \pgfusepath{stroke}
  \end{pgfscope}
  \begin{pgfscope}
    \pgfpathmoveto{\pgfpointxy{650.0}{700.0}}
    \pgfpathlineto{\pgfpointxy{675.0}{725.0}}
    \pgfusepath{stroke}
  \end{pgfscope}
  \begin{pgfscope}
    \pgfpathmoveto{\pgfpointxy{675.0}{725.0}}
    \pgfpatharcaxes{-45.0}{225.0}{\pgfpointxy{35.3553}{0.0}}{\pgfpointxy{0.0}{35.3553}}
    \pgfusepath{stroke}
  \end{pgfscope}
  \begin{pgfscope}
    \pgfpathmoveto{\pgfpointxy{750.0}{700.0}}
    \pgfpathlineto{\pgfpointxy{725.0}{725.0}}
    \pgfusepath{stroke}
  \end{pgfscope}
  \begin{pgfscope}
    \pgfpathmoveto{\pgfpointxy{750.0}{700.0}}
    \pgfpathlineto{\pgfpointxy{775.0}{725.0}}
    \pgfusepath{stroke}
  \end{pgfscope}
  \begin{pgfscope}
    \pgfpathmoveto{\pgfpointxy{775.0}{725.0}}
    \pgfpatharcaxes{-45.0}{225.0}{\pgfpointxy{35.3553}{0.0}}{\pgfpointxy{0.0}{35.3553}}
    \pgfusepath{stroke}
  \end{pgfscope}
  \begin{pgfscope}
    \pgfpathmoveto{\pgfpointxy{850.0}{700.0}}
    \pgfpathlineto{\pgfpointxy{825.0}{725.0}}
    \pgfusepath{stroke}
  \end{pgfscope}
  \begin{pgfscope}
    \pgfpathmoveto{\pgfpointxy{850.0}{700.0}}
    \pgfpathlineto{\pgfpointxy{875.0}{725.0}}
    \pgfusepath{stroke}
  \end{pgfscope}
  \begin{pgfscope}
    \pgfpathmoveto{\pgfpointxy{875.0}{725.0}}
    \pgfpatharcaxes{-45.0}{225.0}{\pgfpointxy{35.3553}{0.0}}{\pgfpointxy{0.0}{35.3553}}
    \pgfusepath{stroke}
  \end{pgfscope}
  \begin{pgfscope}
    \pgfpathmoveto{\pgfpointxy{350.0}{400.0}}
    \pgfpathlineto{\pgfpointxy{325.0}{375.0}}
    \pgfusepath{stroke}
  \end{pgfscope}
  \begin{pgfscope}
    \pgfpathmoveto{\pgfpointxy{350.0}{400.0}}
    \pgfpathlineto{\pgfpointxy{375.0}{375.0}}
    \pgfusepath{stroke}
  \end{pgfscope}
  \begin{pgfscope}
    \pgfpathmoveto{\pgfpointxy{325.0}{375.0}}
    \pgfpatharcaxes{-225.0}{45.0}{\pgfpointxy{35.3553}{0.0}}{\pgfpointxy{0.0}{35.3553}}
    \pgfusepath{stroke}
  \end{pgfscope}
  \begin{pgfscope}
    \pgfpathmoveto{\pgfpointxy{450.0}{400.0}}
    \pgfpathlineto{\pgfpointxy{425.0}{375.0}}
    \pgfusepath{stroke}
  \end{pgfscope}
  \begin{pgfscope}
    \pgfpathmoveto{\pgfpointxy{450.0}{400.0}}
    \pgfpathlineto{\pgfpointxy{475.0}{375.0}}
    \pgfusepath{stroke}
  \end{pgfscope}
  \begin{pgfscope}
    \pgfpathmoveto{\pgfpointxy{425.0}{375.0}}
    \pgfpatharcaxes{-225.0}{45.0}{\pgfpointxy{35.3553}{0.0}}{\pgfpointxy{0.0}{35.3553}}
    \pgfusepath{stroke}
  \end{pgfscope}
  \begin{pgfscope}
    \pgfpathmoveto{\pgfpointxy{550.0}{400.0}}
    \pgfpathlineto{\pgfpointxy{525.0}{375.0}}
    \pgfusepath{stroke}
  \end{pgfscope}
  \begin{pgfscope}
    \pgfpathmoveto{\pgfpointxy{550.0}{400.0}}
    \pgfpathlineto{\pgfpointxy{575.0}{375.0}}
    \pgfusepath{stroke}
  \end{pgfscope}
  \begin{pgfscope}
    \pgfpathmoveto{\pgfpointxy{525.0}{375.0}}
    \pgfpatharcaxes{-225.0}{45.0}{\pgfpointxy{35.3553}{0.0}}{\pgfpointxy{0.0}{35.3553}}
    \pgfusepath{stroke}
  \end{pgfscope}
  \begin{pgfscope}
    \pgfpathmoveto{\pgfpointxy{650.0}{400.0}}
    \pgfpathlineto{\pgfpointxy{625.0}{375.0}}
    \pgfusepath{stroke}
  \end{pgfscope}
  \begin{pgfscope}
    \pgfpathmoveto{\pgfpointxy{650.0}{400.0}}
    \pgfpathlineto{\pgfpointxy{675.0}{375.0}}
    \pgfusepath{stroke}
  \end{pgfscope}
  \begin{pgfscope}
    \pgfpathmoveto{\pgfpointxy{625.0}{375.0}}
    \pgfpatharcaxes{-225.0}{45.0}{\pgfpointxy{35.3553}{0.0}}{\pgfpointxy{0.0}{35.3553}}
    \pgfusepath{stroke}
  \end{pgfscope}
  \begin{pgfscope}
    \pgfpathmoveto{\pgfpointxy{608.97}{693.971}}
    \pgfpatharcaxes{60.0}{90.0}{\pgfpointxy{45.0}{0.0}}{\pgfpointxy{0.0}{45.0}}
    \pgfusepath{stroke}
  \end{pgfscope}
  \begin{pgfscope}
    \pgfpathmoveto{\pgfpointxy{586.47}{700.0}}
    \pgfpatharcaxes{270.0}{300.0}{\pgfpointxy{45.0}{0.0}}{\pgfpointxy{0.0}{45.0}}
    \pgfusepath{stroke}
  \end{pgfscope}
  \begin{pgfscope}
    \pgfpathmoveto{\pgfpointxy{808.97}{693.971}}
    \pgfpatharcaxes{60.0}{90.0}{\pgfpointxy{45.0}{0.0}}{\pgfpointxy{0.0}{45.0}}
    \pgfusepath{stroke}
  \end{pgfscope}
  \begin{pgfscope}
    \pgfpathmoveto{\pgfpointxy{786.47}{700.0}}
    \pgfpatharcaxes{270.0}{300.0}{\pgfpointxy{45.0}{0.0}}{\pgfpointxy{0.0}{45.0}}
    \pgfusepath{stroke}
  \end{pgfscope}
  \begin{pgfscope}
    \pgfpathmoveto{\pgfpointxy{408.97}{393.971}}
    \pgfpatharcaxes{60.0}{90.0}{\pgfpointxy{45.0}{0.0}}{\pgfpointxy{0.0}{45.0}}
    \pgfusepath{stroke}
  \end{pgfscope}
  \begin{pgfscope}
    \pgfpathmoveto{\pgfpointxy{386.47}{400.0}}
    \pgfpatharcaxes{270.0}{300.0}{\pgfpointxy{45.0}{0.0}}{\pgfpointxy{0.0}{45.0}}
    \pgfusepath{stroke}
  \end{pgfscope}
  \begin{pgfscope}
    \pgfpathmoveto{\pgfpointxy{608.97}{393.971}}
    \pgfpatharcaxes{60.0}{90.0}{\pgfpointxy{45.0}{0.0}}{\pgfpointxy{0.0}{45.0}}
    \pgfusepath{stroke}
  \end{pgfscope}
  \begin{pgfscope}
    \pgfpathmoveto{\pgfpointxy{586.47}{400.0}}
    \pgfpatharcaxes{270.0}{300.0}{\pgfpointxy{45.0}{0.0}}{\pgfpointxy{0.0}{45.0}}
    \pgfusepath{stroke}
  \end{pgfscope}
  \begin{pgfscope}
    \pgfsetfillcolor{white}
    \pgfpathellipse{\pgfpointxy{350.0}{700.0}}{\pgfpointxy{8.0}{0.0}}{\pgfpointxy{0.0}{8.0}}
    \pgfusepath{fill,stroke}
  \end{pgfscope}
  \begin{pgfscope}
    \pgfsetfillcolor{white}
    \pgfpathellipse{\pgfpointxy{450.0}{700.0}}{\pgfpointxy{8.0}{0.0}}{\pgfpointxy{0.0}{8.0}}
    \pgfusepath{fill,stroke}
  \end{pgfscope}
  \begin{pgfscope}
    \pgfsetfillcolor{white}
    \pgfpathellipse{\pgfpointxy{550.0}{700.0}}{\pgfpointxy{8.0}{0.0}}{\pgfpointxy{0.0}{8.0}}
    \pgfusepath{fill,stroke}
  \end{pgfscope}
  \begin{pgfscope}
    \pgfsetfillcolor{white}
    \pgfpathellipse{\pgfpointxy{650.0}{700.0}}{\pgfpointxy{8.0}{0.0}}{\pgfpointxy{0.0}{8.0}}
    \pgfusepath{fill,stroke}
  \end{pgfscope}
  \begin{pgfscope}
    \pgfsetfillcolor{white}
    \pgfpathellipse{\pgfpointxy{750.0}{700.0}}{\pgfpointxy{8.0}{0.0}}{\pgfpointxy{0.0}{8.0}}
    \pgfusepath{fill,stroke}
  \end{pgfscope}
  \begin{pgfscope}
    \pgfsetfillcolor{white}
    \pgfpathellipse{\pgfpointxy{850.0}{700.0}}{\pgfpointxy{8.0}{0.0}}{\pgfpointxy{0.0}{8.0}}
    \pgfusepath{fill,stroke}
  \end{pgfscope}
  \begin{pgfscope}
    \pgfsetfillcolor{white}
    \pgfpathellipse{\pgfpointxy{350.0}{400.0}}{\pgfpointxy{8.0}{0.0}}{\pgfpointxy{0.0}{8.0}}
    \pgfusepath{fill,stroke}
  \end{pgfscope}
  \begin{pgfscope}
    \pgfsetfillcolor{white}
    \pgfpathellipse{\pgfpointxy{450.0}{400.0}}{\pgfpointxy{8.0}{0.0}}{\pgfpointxy{0.0}{8.0}}
    \pgfusepath{fill,stroke}
  \end{pgfscope}
  \begin{pgfscope}
    \pgfsetfillcolor{white}
    \pgfpathellipse{\pgfpointxy{550.0}{400.0}}{\pgfpointxy{8.0}{0.0}}{\pgfpointxy{0.0}{8.0}}
    \pgfusepath{fill,stroke}
  \end{pgfscope}
  \begin{pgfscope}
    \pgfsetfillcolor{white}
    \pgfpathellipse{\pgfpointxy{650.0}{400.0}}{\pgfpointxy{8.0}{0.0}}{\pgfpointxy{0.0}{8.0}}
    \pgfusepath{fill,stroke}
  \end{pgfscope}
  \pgftext[bottom,at={\pgfpointxy{350.0}{720.0}}]{$1$}
  \pgftext[bottom,at={\pgfpointxy{450.0}{720.0}}]{$2$}
  \pgftext[bottom,at={\pgfpointxy{550.0}{720.0}}]{$3$}
  \pgftext[bottom,at={\pgfpointxy{650.0}{720.0}}]{$4$}
  \pgftext[bottom,at={\pgfpointxy{750.0}{720.0}}]{$5$}
  \pgftext[bottom,at={\pgfpointxy{850.0}{720.0}}]{$6$}
  \pgftext[top,at={\pgfpointxy{350.0}{380.0}}]{$1$}
  \pgftext[top,at={\pgfpointxy{450.0}{380.0}}]{$2$}
  \pgftext[top,at={\pgfpointxy{550.0}{380.0}}]{$3$}
  \pgftext[top,at={\pgfpointxy{650.0}{380.0}}]{$4$}
  \pgftext[right,at={\pgfpointxy{238.0}{712.5}}]{$\calA$}
  \pgftext[right,at={\pgfpointxy{238.0}{387.5}}]{$\calA'$}
  \pgftext[at={\pgfpointxy{200.0}{550.0}}]{$f$}
\end{pgfpicture}

%% file: tmts.pgf
%% DO NOT FORGET TO \usepackage{pgf,acadpgf}
\begin{pgfpicture}
  \pgfsetxvec{\pgfpoint{\acadpgfunit}{0pt}}
  \pgfsetyvec{\pgfpoint{0pt}{\acadpgfunit}}
  \pgfsetlinewidth{\acadpgflinewidth}
  \pgftransformshift{\pgfpointxy{25.0}{-50.0}}

  \begin{pgfscope}
    \pgfpathmoveto{\pgfpointxy{500.0}{350.0}}
    \pgfpathlineto{\pgfpointxy{350.0}{436.603}}
    \pgfusepath{stroke}
  \end{pgfscope}
  \begin{pgfscope}
    \pgfpathmoveto{\pgfpointxy{350.0}{436.603}}
    \pgfpathlineto{\pgfpointxy{350.0}{263.397}}
    \pgfusepath{stroke}
  \end{pgfscope}
  \begin{pgfscope}
    \pgfpathmoveto{\pgfpointxy{350.0}{263.397}}
    \pgfpathlineto{\pgfpointxy{500.0}{350.0}}
    \pgfusepath{stroke}
  \end{pgfscope}
  \begin{pgfscope}
    \pgfpathmoveto{\pgfpointxy{1000.0}{350.0}}
    \pgfpathlineto{\pgfpointxy{1150.0}{263.397}}
    \pgfusepath{stroke}
  \end{pgfscope}
  \begin{pgfscope}
    \pgfpathmoveto{\pgfpointxy{1150.0}{263.397}}
    \pgfpathlineto{\pgfpointxy{1150.0}{436.603}}
    \pgfusepath{stroke}
  \end{pgfscope}
  \begin{pgfscope}
    \pgfpathmoveto{\pgfpointxy{1150.0}{436.603}}
    \pgfpathlineto{\pgfpointxy{1000.0}{350.0}}
    \pgfusepath{stroke}
  \end{pgfscope}
  \begin{pgfscope}
    \pgfpathmoveto{\pgfpointxy{550.0}{362.5}}
    \pgfpathlineto{\pgfpointxy{700.0}{362.5}}
    \pgfusepath{stroke}
  \end{pgfscope}
  \begin{pgfscope}
    \pgfpathmoveto{\pgfpointxy{800.0}{362.5}}
    \pgfpathlineto{\pgfpointxy{950.0}{362.5}}
    \pgfusepath{stroke}
  \end{pgfscope}
  \begin{pgfscope}
    \pgfpathmoveto{\pgfpointxy{550.0}{337.5}}
    \pgfpathlineto{\pgfpointxy{700.0}{337.5}}
    \pgfusepath{stroke}
  \end{pgfscope}
  \begin{pgfscope}
    \pgfpathmoveto{\pgfpointxy{800.0}{337.5}}
    \pgfpathlineto{\pgfpointxy{950.0}{337.5}}
    \pgfusepath{stroke}
  \end{pgfscope}
  \begin{pgfscope}
    \pgfpathmoveto{\pgfpointxy{650.0}{350.0}}
    \pgfpathlineto{\pgfpointxy{600.0}{375.0}}
    \pgfusepath{stroke}
  \end{pgfscope}
  \begin{pgfscope}
    \pgfpathmoveto{\pgfpointxy{650.0}{350.0}}
    \pgfpathlineto{\pgfpointxy{600.0}{325.0}}
    \pgfusepath{stroke}
  \end{pgfscope}
  \begin{pgfscope}
    \pgfpathmoveto{\pgfpointxy{900.0}{350.0}}
    \pgfpathlineto{\pgfpointxy{850.0}{375.0}}
    \pgfusepath{stroke}
  \end{pgfscope}
  \begin{pgfscope}
    \pgfpathmoveto{\pgfpointxy{900.0}{350.0}}
    \pgfpathlineto{\pgfpointxy{850.0}{325.0}}
    \pgfusepath{stroke}
  \end{pgfscope}
  \begin{pgfscope}
    \pgfpathmoveto{\pgfpointxy{750.0}{475.0}}
    \pgfpatharcaxes{90.0}{118.61}{\pgfpointxy{625.0}{0.0}}{\pgfpointxy{0.0}{625.0}}
    \pgfusepath{stroke}
  \end{pgfscope}
  \begin{pgfscope}
    \pgfpathmoveto{\pgfpointxy{750.0}{500.0}}
    \pgfpatharcaxes{90.0}{119.476}{\pgfpointxy{650.0}{0.0}}{\pgfpointxy{0.0}{650.0}}
    \pgfusepath{stroke}
  \end{pgfscope}
  \begin{pgfscope}
    \pgfpathmoveto{\pgfpointxy{1049.28}{398.685}}
    \pgfpatharcaxes{61.3895}{90.0}{\pgfpointxy{625.0}{0.0}}{\pgfpointxy{0.0}{625.0}}
    \pgfusepath{stroke}
  \end{pgfscope}
  \begin{pgfscope}
    \pgfpathmoveto{\pgfpointxy{1069.84}{415.866}}
    \pgfpatharcaxes{60.5241}{90.0}{\pgfpointxy{650.0}{0.0}}{\pgfpointxy{0.0}{650.0}}
    \pgfusepath{stroke}
  \end{pgfscope}
  \begin{pgfscope}
    \pgfpathmoveto{\pgfpointxy{775.0}{487.5}}
    \pgfpathlineto{\pgfpointxy{725.0}{512.5}}
    \pgfusepath{stroke}
  \end{pgfscope}
  \begin{pgfscope}
    \pgfpathmoveto{\pgfpointxy{775.0}{487.5}}
    \pgfpathlineto{\pgfpointxy{725.0}{462.5}}
    \pgfusepath{stroke}
  \end{pgfscope}
  \begin{pgfscope}
    \pgfpathmoveto{\pgfpointxy{417.5}{309.33}}
    \pgfpatharcaxes{270.0}{300.0}{\pgfpointxy{45.0}{0.0}}{\pgfpointxy{0.0}{45.0}}
    \pgfusepath{stroke}
  \end{pgfscope}
  \begin{pgfscope}
    \pgfpathmoveto{\pgfpointxy{440.0}{315.359}}
    \pgfpatharcaxes{120.0}{150.0}{\pgfpointxy{45.0}{0.0}}{\pgfpointxy{0.0}{45.0}}
    \pgfusepath{stroke}
  \end{pgfscope}
  \begin{pgfscope}
    \pgfpathmoveto{\pgfpointxy{356.029}{355.179}}
    \pgfpatharcaxes{150.0}{180.0}{\pgfpointxy{45.0}{0.0}}{\pgfpointxy{0.0}{45.0}}
    \pgfusepath{stroke}
  \end{pgfscope}
  \begin{pgfscope}
    \pgfpathmoveto{\pgfpointxy{350.0}{332.679}}
    \pgfpatharcaxes{0.0}{30.0}{\pgfpointxy{45.0}{0.0}}{\pgfpointxy{0.0}{45.0}}
    \pgfusepath{stroke}
  \end{pgfscope}
  \begin{pgfscope}
    \pgfpathmoveto{\pgfpointxy{426.471}{385.49}}
    \pgfpatharcaxes{30.0}{60.0}{\pgfpointxy{45.0}{0.0}}{\pgfpointxy{0.0}{45.0}}
    \pgfusepath{stroke}
  \end{pgfscope}
  \begin{pgfscope}
    \pgfpathmoveto{\pgfpointxy{410.0}{401.962}}
    \pgfpatharcaxes{240.0}{270.0}{\pgfpointxy{45.0}{0.0}}{\pgfpointxy{0.0}{45.0}}
    \pgfusepath{stroke}
  \end{pgfscope}
  \begin{pgfscope}
    \pgfpathmoveto{\pgfpointxy{1073.53}{314.51}}
    \pgfpatharcaxes{210.0}{240.0}{\pgfpointxy{45.0}{0.0}}{\pgfpointxy{0.0}{45.0}}
    \pgfusepath{stroke}
  \end{pgfscope}
  \begin{pgfscope}
    \pgfpathmoveto{\pgfpointxy{1090.0}{298.038}}
    \pgfpatharcaxes{60.0}{90.0}{\pgfpointxy{45.0}{0.0}}{\pgfpointxy{0.0}{45.0}}
    \pgfusepath{stroke}
  \end{pgfscope}
  \begin{pgfscope}
    \pgfpathmoveto{\pgfpointxy{1143.97}{344.821}}
    \pgfpatharcaxes{-30.0}{0.0}{\pgfpointxy{45.0}{0.0}}{\pgfpointxy{0.0}{45.0}}
    \pgfusepath{stroke}
  \end{pgfscope}
  \begin{pgfscope}
    \pgfpathmoveto{\pgfpointxy{1150.0}{367.321}}
    \pgfpatharcaxes{180.0}{210.0}{\pgfpointxy{45.0}{0.0}}{\pgfpointxy{0.0}{45.0}}
    \pgfusepath{stroke}
  \end{pgfscope}
  \begin{pgfscope}
    \pgfpathmoveto{\pgfpointxy{1082.5}{390.67}}
    \pgfpatharcaxes{90.0}{120.0}{\pgfpointxy{45.0}{0.0}}{\pgfpointxy{0.0}{45.0}}
    \pgfusepath{stroke}
  \end{pgfscope}
  \begin{pgfscope}
    \pgfpathmoveto{\pgfpointxy{1060.0}{384.641}}
    \pgfpatharcaxes{300.0}{330.0}{\pgfpointxy{45.0}{0.0}}{\pgfpointxy{0.0}{45.0}}
    \pgfusepath{stroke}
  \end{pgfscope}
  \begin{pgfscope}
    \pgfpathmoveto{\pgfpointxy{175.0}{300.0}}
    \pgfpathlineto{\pgfpointxy{25.0}{300.0}}
    \pgfusepath{stroke}
  \end{pgfscope}
  \begin{pgfscope}
    \pgfpathmoveto{\pgfpointxy{92.5}{306.029}}
    \pgfpatharcaxes{240.0}{270.0}{\pgfpointxy{45.0}{0.0}}{\pgfpointxy{0.0}{45.0}}
    \pgfusepath{stroke}
  \end{pgfscope}
  \begin{pgfscope}
    \pgfpathmoveto{\pgfpointxy{115.0}{300.0}}
    \pgfpatharcaxes{90.0}{120.0}{\pgfpointxy{45.0}{0.0}}{\pgfpointxy{0.0}{45.0}}
    \pgfusepath{stroke}
  \end{pgfscope}
  \begin{pgfscope}
    \pgfsetfillcolor{white}
    \pgfpathellipse{\pgfpointxy{750.0}{350.0}}{\pgfpointxy{8.0}{0.0}}{\pgfpointxy{0.0}{8.0}}
    \pgfusepath{fill,stroke}
  \end{pgfscope}
  \begin{pgfscope}
    \pgfsetfillcolor{white}
    \pgfpathellipse{\pgfpointxy{500.0}{350.0}}{\pgfpointxy{8.0}{0.0}}{\pgfpointxy{0.0}{8.0}}
    \pgfusepath{fill,stroke}
  \end{pgfscope}
  \begin{pgfscope}
    \pgfsetfillcolor{white}
    \pgfpathellipse{\pgfpointxy{1000.0}{350.0}}{\pgfpointxy{8.0}{0.0}}{\pgfpointxy{0.0}{8.0}}
    \pgfusepath{fill,stroke}
  \end{pgfscope}
  \begin{pgfscope}
    \pgfsetfillcolor{white}
    \pgfpathellipse{\pgfpointxy{350.0}{436.603}}{\pgfpointxy{8.0}{0.0}}{\pgfpointxy{0.0}{8.0}}
    \pgfusepath{fill,stroke}
  \end{pgfscope}
  \begin{pgfscope}
    \pgfsetfillcolor{white}
    \pgfpathellipse{\pgfpointxy{350.0}{263.397}}{\pgfpointxy{8.0}{0.0}}{\pgfpointxy{0.0}{8.0}}
    \pgfusepath{fill,stroke}
  \end{pgfscope}
  \begin{pgfscope}
    \pgfsetfillcolor{white}
    \pgfpathellipse{\pgfpointxy{1150.0}{263.397}}{\pgfpointxy{8.0}{0.0}}{\pgfpointxy{0.0}{8.0}}
    \pgfusepath{fill,stroke}
  \end{pgfscope}
  \begin{pgfscope}
    \pgfsetfillcolor{white}
    \pgfpathellipse{\pgfpointxy{1150.0}{436.603}}{\pgfpointxy{8.0}{0.0}}{\pgfpointxy{0.0}{8.0}}
    \pgfusepath{fill,stroke}
  \end{pgfscope}
  \begin{pgfscope}
    \pgfsetfillcolor{white}
    \pgfpathellipse{\pgfpointxy{25.0}{300.0}}{\pgfpointxy{8.0}{0.0}}{\pgfpointxy{0.0}{8.0}}
    \pgfusepath{fill,stroke}
  \end{pgfscope}
  \begin{pgfscope}
    \pgfsetfillcolor{white}
    \pgfpathellipse{\pgfpointxy{175.0}{300.0}}{\pgfpointxy{8.0}{0.0}}{\pgfpointxy{0.0}{8.0}}
    \pgfusepath{fill,stroke}
  \end{pgfscope}
  \pgftext[top,at={\pgfpointxy{500.0}{330.0}}]{1}
  \pgftext[bottom,right,at={\pgfpointxy{336.687}{450.584}}]{2}
  \pgftext[top,right,at={\pgfpointxy{338.352}{248.278}}]{3}
  \pgftext[top,at={\pgfpointxy{750.0}{330.0}}]{4}
  \pgftext[top,at={\pgfpointxy{1000.0}{330.0}}]{5}
  \pgftext[top,left,at={\pgfpointxy{1162.9}{249.071}}]{6}
  \pgftext[bottom,left,at={\pgfpointxy{1162.96}{450.879}}]{7}
  \pgftext[bottom,at={\pgfpointxy{750.0}{212.0}}]{$\calB$}
  \pgftext[bottom,at={\pgfpointxy{25.0}{320.0}}]{1}
  \pgftext[bottom,at={\pgfpointxy{175.0}{320.0}}]{2}
  \pgftext[bottom,at={\pgfpointxy{100.0}{212.0}}]{$\calA$}
\end{pgfpicture}

%% file: DTHMS-23.bbl
\begin{thebibliography}{99}
\bibitem{AH}
  F.\ G.\ Abramson, L.\ A.\ Harrington.
  Models without indiscernibles.
  J.~Symbolic Logic 43 (1978), 572--600.

\bibitem{AHS}
  J.\ Ad\'amek, H.\ Herrlich, G.\ E.\ Strecker.
  Abstract and Concrete Categories: The Joy of Cats.
  Dover Books on Mathematics, Dover Publications 2009

\bibitem{MANY-MANY}
  A.\ Aranda, D.\ Bradley-Williams, J.\ Hubi\v cka, M.\ Karamanlis, M.\ Kompatscher, M.\ Kone\v cn\'y, M.\ Pawliuk.
  Ramsey expansions of metrically homogeneous graphs.
  arXiv:1707.02612

\bibitem{frankl-graham-rodl}
  P. Frankl, R. L. Graham, V. R\"odl.
  Induced restricted Ramsey theorems for spaces.
  Journal of Combinatorial Theory Ser. A, 44 (1987), 120--128.

\bibitem{GR}
  R.\ L.\ Graham, B.\ L.\ Rothschild.
  Ramsey's theorem for n-parameter sets.
  Tran.\ Amer.\ Math.\ Soc.\ 159 (1971), 257--292.

\bibitem{masul-preadj}
  D.\ Ma\v sulovi\'c.
  Pre-adjunctions and the Ramsey property.
  European Journal of Combinatorics, 70 (2018), 268--283.

\bibitem{masul-drp-perm}
  D.\ Ma\v sulovi\'c.
  A Dual Ramsey Theorem for Permutations.
  Electronic Journal of Combinatorics 24(3) (2017), \#P3.39

\bibitem{masul-mudri}
  D.\ Ma\v sulovi\'c, N.\ Mudrinski.
  On the dual Ramsey property for finite distributive lattices.
  Order 34 (2017), 479--490.

\bibitem{masulovic-ramsey}
  D.\ Ma\v sulovi\'c, L.\ Scow.
  Categorical equivalence and the Ramsey property for finite powers of a primal algebra.
  Algebra Universalis 78 (2017), 159--179.

\bibitem{N1995}
  J.\ Ne\v set\v ril.
  Ramsey theory. In: R.\ L.\ Graham, M.\ Gr\"otschel and L.\ Lov\'asz, eds, Handbook of Combinatorics, Vol.~2,
  1331--1403, MIT Press, Cambridge, MA, USA, 1995.

\bibitem{Nesetril-metric}
  J.\ Ne\v set\v ril.
  Metric spaces are Ramsey.
  European Journal of Combinatorics 28 (2007), 457--468.

\bibitem{Nesetril-Rodl}
  J.\ Ne\v set\v ril, V.\ R\"odl.
  Partitions of finite relational and set systems.
  J.\ Combin.\ Theory Ser.\ A 22 (1977), 289--312.

\bibitem{Nesetril-Rodl-DRT}
  J.\ Ne\v set\v ril, V.\ R\"odl.
  Dual Ramsey type theorems.
  In: Z.\ Frol\'ik (ed), Proc.\ Eighth Winter School on Abstract Analysis, Prague, 1980, 121--123.

\bibitem{Nesetril-Rodl-1983}
  J.\ Ne\v set\v ril, V.\ R\"odl.
  Ramsey classes of set systems.
  J.~Combinat.\ Theory A 34 (1983), 183--201

\bibitem{Nesetril-Rodl-1989}
  J.\ Ne\v set\v ril, V.\ R\"odl.
  The partite construction and Ramsey set systems.
  Discr.\ Math.\ 75 (1989), 327--334


\bibitem{Promel-1985}
  H.\ J.\ Pr\"omel.
  Induced partition properties of combinatorial cubes.
  J.~Combinat.\ Theory A 39 (1985), 177--208

\bibitem{promel-voigt-surj-sets}
  H.\ J.\ Pr\"omel, B.\ Voigt. 
  Hereditary atributes of surjections and parameter sets.
  Europ.~J.~Combinatorics 7 (1986), 161–170.

\bibitem{promel-voigt-sparse-GR}
  H.\ J.\ Pr\"omel, B.\ Voigt. 
  A sparse Graham-Rothschild theorem.
  Transactions of the American Mathematical Society 309.1 (1988), 113--137.

\bibitem{Ramsey}
  F.\ P.\ Ramsey.
  On a problem of formal logic.
  Proc.\ London Math.\ Soc.\ 30 (1930), 264--286.

\bibitem{sokic2}
  M.\ Soki\'c.
  Ramsey Properties of Finite Posets II.
  Order, 29 (2012) 31--47.

\bibitem{Solecki-2010}
  S. Solecki.
  A Ramsey theorem for structures with both relations and functions.
  Journal of Combinatorial Theory, Ser. A, 117 (2010), 704--714.

\bibitem{solecki-dual-ramsey-trees}
  S.\ Solecki.
  Dual Ramsey theorem for trees.
  Preprint, arXiv:1502.04442

\bibitem{spencer}
  J.\ H.\ Spencer.
  Ramsey’s theorem for spaces.
  Transactions of the American Mathematical Society 249.2 (1979), 363--371.
\end{thebibliography}
